\documentclass[12pt, reqno]{amsart}
\usepackage{amsmath}
\usepackage{amssymb}
\usepackage{epsfig}
\usepackage{graphicx}
\usepackage{color}
\usepackage{fullpage}
\definecolor{shadecolor}{gray}{0.875}
\usepackage{amscd}
\usepackage{comment}

\numberwithin{equation}{section}

\newcommand{\sho}[1]{{\color{red} \sf $\clubsuit\clubsuit\clubsuit$ Sho: [#1]}}

\input xy
\xyoption{all}

\calclayout
\allowdisplaybreaks[3]

\theoremstyle{plain}
\newtheorem{prop}{Proposition}[section]

\newtheorem{theo}[prop]{Theorem}

\newtheorem{lemm}[prop]{Lemma}

\theoremstyle{definition}
\newtheorem{defi}[prop]{Definition}

\newtheorem{conj}[prop]{Conjecture}

\newtheorem{rema}[prop]{Remark}

\newtheorem{exam}[prop]{Example}

\makeatother
\makeatletter

\author{Sho Tanimoto}
\address{Department of Mathematics, Faculty of Science, Kumamoto University, Kurokami 2-39-1 Kumamoto 860-8555 Japan}
\address{Priority Organization for Innovation and Excellence, Kumamoto University}
\email{stanimoto@kumamoto-u.ac.jp}

\title[Upper bounds of Manin type]{On upper bounds of Manin type}

\begin{document}

\begin{abstract}
We introduce a certain birational invariant of a polarized algebraic variety and use that to obtain upper bounds for the counting functions of rational points on algebraic varieties. Using our theorem, we obtain new upper bounds of Manin type for $28$ deformation types of smooth Fano $3$-folds of Picard rank $\geq 2$ following Mori-Mukai's classification.
We also find new upper bounds for polarized K3 surfaces $S$ of Picard rank $1$ using Bayer-Macr\`i's result on the nef cone of the Hilbert scheme of two points on $S$.
\end{abstract}

\maketitle

\date{\today}

\section{Introduction} 
\label{sect:intro}

A driving question in diophantine geometry is to prove asymptotic formulae for the counting function of rational points on a projective variety. Manin's conjecture, originally formulated in \cite{BM}, predicts a precise asymptotic formula when the underlying variety is smooth Fano, or more generally smooth and rationally connected. This asymptotic formula has a description in terms of the geometric invariants of the underlying variety.

In this paper we consider questions related to the following weaker version of the conjecture which is called weak Manin's conjecture: let $X$ be a geometrically uniruled smooth projective variety defined over a number field $k$ and let $L$ be a big and nef divisor on $X$. One can associate a height function
\[
H_{L} : X(k) \rightarrow \mathbb R_{>0},
\]
to $(X, L)$, and we consider the following counting function:
\[
N(U, L, T) = \#\{ P \in U(k) | H_L(P) \leq T\}
\]
for an appropriate Zariski open subset $U\subset X$. Weak Manin's conjecture predicts that this function is governed by the following geometric invariant of $(X, L)$:
\[
a(X, L) = \inf\{ t \in \mathbb R \mid K_X + tL \in \overline{\mathrm{Eff}}^1(X)\},
\]
where $\overline{\mathrm{Eff}}^1(X)$ is the cone of pseudo-effective divisors on $X$. 
Here is the statement of weak Manin's conjecture:
\begin{conj}[Weak Manin's conjecture/Linear growth conjecture]
\label{conj:weakManin}
Let $X$ be a geometrically uniruled smooth projective variety defined over a number field $k$ and let $L$ be a big and nef divisor on $X$. Then there exists a non-empty Zariski open subset $U \subset X$ such that for any $\epsilon > 0$
\[
N(U, L, T) = O_\epsilon(T^{a(X, L) +\epsilon}).
\]
Note that for $L =-K_X$, we have $a(X, L) =1$ so that it backs up the word ``linear growth''.
\end{conj}

\begin{rema}
There are counterexamples to a version of this conjecture where one assumes $L$ to be only big but not nef. See \cite[Section 5.1]{LST18}
\end{rema}

The starting point of the current research is \cite{Mck11}. In this paper, McKinnon shows that Vojta's conjecture implies weak Manin's conjecture for K3 surfaces and more generally varieties with Kodaira dimension $0$ assuming the Non-Vanishing conjecture in the minimal model program. (For such varieties, the $a$-invariant is $0$.) While McKinnon's result is conditional on Vojta's conjecture, our results are unconditional: they do not rely on Vojta's conjecture. In our approach, instead of appealing to Vojta's conjecture, we use the positivity of divisors by introducing the following invariant measuring the local positivity of big divisors:

\begin{defi}
Let $X$ be a normal projective variety defined over an algebraically closed field of characteristic $0$ and $H$ be a big $\mathbb Q$-Cartier divisor on $X$. 
We consider $W = X \times X$ and denote each projection by $\pi_i : W \rightarrow X_i$.
Let $\alpha : W' \rightarrow W$ be the blow up of the diagonal and denote its exceptional divisor by $E$, i.e., the pullback of the diagonal.
For any $\mathbb Q$-Cartier divisor $L$ on $X$ we denote $\alpha^*\pi_1^*L + \alpha^*\pi_2^*L$ by $L[2]$.
We define the following invariant 
\[
\delta(X, H) = \inf\left\{s\in \mathbb R \middle| \begin{matrix}\text{ for any component $V \subset \mathrm{SB}(sH[2]- E)$ not contained in $E$,} \\
\text{one of $\pi_i\circ \alpha|_V$ is not dominant to $X_i$. }
\end{matrix}\right\},
\] 
where $\mathrm{SB}(sH[2]- E)$ is the stable base locus of a $\mathbb R$-divisor $sH[2]- E$. We call this invariant the $\delta$-invariant.
\end{defi}

Inspired by \cite{Mck11}, we obtain the following general result on the counting functions of rational points on algebraic varieties:
\begin{theo}
\label{theo:generalintro}
Let $X$ be a normal projective variety of dimension $n$ defined over a number field $k$ and $L$ be a big $\mathbb Q$-Cartier divisor on $X$. 
Then for any $\epsilon > 0$ there exists a non-empty Zariski open subset $U = U(\epsilon) \subset X$ such that  we have
\[
N(U, L, T) = O_\epsilon(T^{2n\delta(X, L) + \epsilon}).
\]
\end{theo}
Previous results applying to general projective varieties are results related to dimension growth conjecture obtained by Browning, Heath-Brown, and Salberger (\cite{BHBS06}), and Salberger (\cite{Sal07}). Recall that the dimension growth conjecture of Heath-Brown, which is proved by Salberger, states that for any subvariety $X \subset \mathbb P^n$, the hyperplane class $H$, and any $\epsilon > 0$ we have 
\[
N(X, H, T) = O(T^{\dim X + \epsilon}).
\]
For many examples of projective varietes $X$ of dimension $\leq 3$ we proved $\delta(X, -K_X) \leq \frac{1}{2}$. (For example, we compute $\delta(X, -K_X)$ for most of $3$ dimensional Fano conic bundles with rational sections in Section~\ref{sec:Fanoconic} following Mori-Mukai's classification, and we confirm that $\delta(X, -K_X) = \frac{1}{2}$ for $40$ deformation types out of $53$ deformation types.) Thus this theorem recovers some statements of \cite{BHBS06} and \cite{Sal07} on the dimension growth conjecture for varieties with $\delta(X, -K_X) \leq 1/2$. We conjecture that for a large portion of the class of Fano manifolds, we have $\delta(X, -K_X) \leq 1/2$ so our theorem should lead to an alternative proof of the dimension growth conjecture for certain Fano varieties.

However, there is no direct comparison between our result and results in \cite{BHBS06} and \cite{Sal07}. \cite{BHBS06} and \cite{Sal07} are certainly better in the sense that they obtain a bound for $N(X, L, T)$ and their constants only depend on the dimension of $X$, $\epsilon$, and the dimension of the ambient projective space where $X$ is embedded into. On the other hand, our method also has the advantage in the sense that our theorem applies to arbitrary big divisor and in many cases where $\delta(X, L)$ is the minimum, e.g., $3$-dimensional Fano conic bundles mentioned above, one does not need to introduce $\epsilon > 0$ in the above theorem. 
For most smooth projective varieties with non-negative Kodaira dimension, we conjecture that Theorem~\ref{theo:generalintro} gives better upper bounds than \cite{BHBS06} and \cite{Sal07}. For example, we have the following application of Theorem~\ref{theo:generalintro} to K3 surfaces:
\begin{theo}
\label{theo:K3}
Let $S$ be a K3 surface defined over a number field $k$ with a polarization $H$ of degree $2d$ such that $\mathrm{Pic}(\overline{S}) = \mathbb ZH$. Then for any $\epsilon > 0$, we have
\[
N(S, H, T) = O_\epsilon(T^{4\sqrt{\frac{4}{d} + \frac{5}{d^2}} + \epsilon}).
\]
\end{theo}
The existence of K3 surfaces satisfying the assumptions of the above theorem is justified by a series of papers \cite{Tera85}, \cite{Ell04}, and \cite{vL07}.
Our proof of the above theorem is relied on the work of \cite{BM14} on the nef cone of the Hilbert scheme of $2$ points $\mathrm{Hilb}^{[2]}(S)$. 
Indeed, $\delta(S, H)$ is bounded by the $s$-invariant of $H$, and the computation of the $s$-invariant can be done using the description of the nef cone of $\mathrm{Hilb}^{[2]}(S)$. 
A certain bound is also obtained for Enriques surfaces by using the work \cite{Nuer}.
See Theorem~\ref{theo:enriques}.
It would be interesting to compute these invariants for surfaces of general type. Bounds of this type are obtained for hypersurfaces in $\mathbb P^n$ by Heath-Brown \cite{HB02}.

For some $3$ dimensional Fano conic bundles we can improve bounds of Theorem~\ref{theo:generalintro} using conic bundle structures.

\begin{theo}
\label{theo:mainintro}
Let $f: X \rightarrow S$ be a conic bundle defined over a number field $k$ with a rational section.
We assume that $X$ and $S$ are smooth Fano. 
Let $W = X \times X$ and $W'$ be the blow up of $W$ along the diagonal with the exceptional divisor $E$. We denote each projection $W' \rightarrow X_i$ by $\pi_i$. 
Let $\alpha, \beta$ be positive real numbers such that $2\alpha- 2\beta = 1$. We further make the following assumptions:
\begin{enumerate}
\item Weak Manin's conjecture for $(S, -K_S)$ holds,
\item for any component $V$ of the stable locus of the following divisor
\[
-\alpha K_{X/S}[2] - \beta f^*K_S[2] -E,
\]
such that $V$ is not contained in $E$, one of projections $\pi_i|_V$ is not dominant.
\end{enumerate}
Then there exists a non-empty Zariski open subset $U \subset X$ such that for any $\epsilon > 0$, there exists $C = C_\epsilon >0$ such that
\[
N(U, -K_X, T) < CT^{2\alpha +  \epsilon}.
\]
\end{theo}

Using this theorem, new upper bounds for $28$ deformation types of Fano $3$-folds are obtained and these bounds are better than the dimension growth conjecture in \cite{BHBS06} and \cite{Sal07}. These examples are discussed in Section~\ref{sec:Fanoconic}. Here are some examples of Fano $3$-folds which our theorem applies to. Note that for examples below we have $\delta(X, -K_X) =1/2$.

\begin{exam}[Example~\ref{exam:no31}]
Let $X$ be the blow-up of a quadric threefold $Q$ defined over a number field $k$ with center a line defined over the same ground field. Let $H$ be the pullback of hyperplane class from $Q$ and we denote the exceptional divisor by $D$. Then the linear system $|H-D|$ defines a $\mathbb P^1$-fibration over $\mathbb P^2$. We prove that $\alpha = 5/6$ satisfies the assumptions of Theorem~\ref{theo:mainintro}: thus we may conclude there exists some open subset $U \subset X$ such that for any $\epsilon > 0$, we have
\[
N(U, -K_X, T) = O_\epsilon(T^{5/3 + \epsilon}).
\]
\end{exam}

\begin{exam}[Example~\ref{exam:no23}]
Let $V_7$ be the blow-up of $\mathbb P^3$ at a point $P$. This is isomorphic to $\mathbb P(\mathcal O \oplus \mathcal O(1))$ over $\mathbb P^2$. Let $X$ be the blow-up of $V_7$ with center the strict transform of a conic passing through $P$. Then $X$ is a Fano conic bundle with singular fibers. We prove that $\alpha = 5/6$ satisfies the assumptions of Theorem~\ref{theo:mainintro}: thus we may conclude there exists some open subset $U \subset X$ such that for any $\epsilon > 0$, we have
\[
N(U, -K_X, T) = O_\epsilon(T^{5/3 + \epsilon}).
\]

\end{exam}

\begin{exam}[Example~\ref{exam:no6}]
Let $X$ be the blow-up of $\mathbb P^3$ with center a disjoint union of three lines. Then $X$ is a Fano conic bundle with singular fibers. We prove that $\alpha = 1$ satisfies the assumptions of Theorem~\ref{theo:mainintro}: thus we may conclude there exists some open subset $U \subset X$ such that for any $\epsilon > 0$, we have
\[
N(U, -K_X, T) = O_\epsilon(T^{2 + \epsilon}).
\]

\end{exam}


It is natural to wonder whether $2\alpha < 2 \delta(X, -K_X) \dim X$ holds in general. While we do not have a proof of this inequality, we do not have any counterexample neither. Finally note that for del Pezzo surfaces, there are many better results on bounds of the counting functions, see, e.g., \cite{HB97}, \cite{Bro01}, \cite{BSJ14}, \cite{FLS16}, and \cite{BS18}.





\subsection{The method of proofs}
McKinnon proves weak Manin's conjecture for K3 surfaces using a certain Repulsion principle which he proves assuming Vojta's conjecture. We instead prove a different Repulsion principle using the $\delta$-invariant and this proof does not rely on Vojta's conjecture. Here is our theorem:
\begin{theo}[Repulsion principle]
\label{theo:repulsionintro}
Let $X$ be a normal projective variety defined over a number field $k$. We fix a place $v$ of $k$. Let $A$ be a big $\mathbb Q$-Cartier divisor on $X$. Then for any $\epsilon > 0$ there exists a constant $C= C_\epsilon > 0$ and a non-empty Zariski open subset $U = U(\epsilon) \subset X$ such that we have
\[
\mathrm{dist}_v(P, Q) > C (H_A(P)H_A(Q))^{-(\delta(X, A) + \epsilon)},
\]
for any $P, Q \in U(k)$ with $P \neq Q$, where $\mathrm{dist}_v(P, Q)$ is the $v$-adic distant function on $X$.
\end{theo}
Combining this theorem and counting arguments in \cite{Mck11}, we prove Theorem~\ref{theo:generalintro}.

\

Here is the road map of this paper: in Section~\ref{sec:height}, we recall the constructions of height functions and their basic properties.  In Section~\ref{sect:fujita} we recall basic properties and results of the $a$-invariants. In Section~\ref{sec:delta}, we discuss some basic properties of the $\delta$-invariants and compute them for some examples, e.g., del Pezzo surfaces. In Section~\ref{sec:repulsion}, we prove the Repulsion principle for projective varieties (Theorem~\ref{theo:repulsionintro}). In Section~\ref{sect:general} we establish Theorem~\ref{theo:generalintro}. In Section~\ref{sec:K3}, we study K3 surfaces and Enriques surfaces and prove Theorem~\ref{theo:K3}. In Section~\ref{sect:countingonconicbundle}, we prove Theorem~\ref{theo:mainintro}. In Section~\ref{sec:Fanoconic}, we study $3$-dimensional Fano conic bundles using Theorem~\ref{theo:mainintro}. 

\

\noindent
{\bf Acknowledgement}
The author would like to thank Brian Lehmann for helpful conversations, and careful reading and comments on an early draft of this paper. In particular the author thanks Brian for his suggestions regarding a relation of the $\delta$-invariant to the Seshadri constant and the $s$-invariant, and applications of his works to K3 surfaces and Enriques surfaces. The author also would like to thank Takeshi Abe for answering his question regarding the second chern form. The author would like to thank Shigefumi Mori and Shigeru Mukai for teaching him about the work of Matsuki on the cone of curves on Fano $3$-folds. (\cite{Mat95}) The author thanks Tim Browning, Zhizhong Huang, and Yuri Tschinkel for comments on this paper. The author thanks anonymous referees for suggestions to improve the exposition of the paper.
Sho Tanimoto is partially supported by MEXT Japan, Leading Initiative for Excellent Young Researchers (LEADER), Inamori Foundation, and JSPS KAKENHI Early-Career Scientists Grant numbers 19K14512.

\section{Height functions}
\label{sec:height}

Here we recall the constructions of height functions and their basic properties which will be needed for the rest of the paper. The main references are \cite{Silver00} as well as \cite{volume}.
Let $k$ be a number field and $M_k$ denote the set of places of $k$. For each place $v \in M_k$, $k_v$ denotes its completion with respect to $v$, and we fix a haar measure $\mu_v$ on $k_v$. We normalize our absolute value $|\cdot|_v$ on $k_v$ by the following property: for any $a \in k_v$ and a measurable set $\Omega \subset k_v$ we have
\[
\mu_v(a\Omega) = |a|_v \mu_v(\Omega). 
\]
When $k_v = \mathbb R$, $|\cdot|_v$ is the usual absolute value. When $k_v = \mathbb Q_p$ we have $|p|_v = 1/p$. For a finite extension $k_v/\mathbb Q_w$, we have $|a|_v = |N_{k_v/\mathbb Q_w}(a)|_w$ for any $a \in k_v$. Due to the normalizations, we have the product formual, i.e., for any $a \in k^\times$ we have
\begin{equation}
\label{eqn:product}
\prod_{v \in M_k} |a|_v = 1.
\end{equation}

A variety $X$ defined over $k$ is a geometrically integral separated scheme of finite type over $k$.
For any place $v \in M_k$ the topological space $X(k_v)$ is endowed with a natural structure as an analytic space over $k_v$. For any invertible sheaf $L$ on $X$, we consider the underlying line bundle $\pi : \underline{L} \rightarrow X$.

\begin{defi}[Height functions]
Let $X$ be a projective variety defined over $k$ and $L$ be an invertible sheaf on $X$. Let $\{\|\cdot\|_v\}_{v\in M_k}$ be an adelic metric for $L$, i.e., for each $v \in M_k$, $\|\cdot\|_v$ is a $v$-adic metric on the analytic line bundle $\pi: \underline{L}(k_v) \rightarrow X(k_v)$ and they satisfy a certain integral condition. (See \cite[Section 2.2.3]{volume} for more details.) For each $P \in X(k)$, we pick a non-zero element $\ell$ of $\underline{L}_P(k)$ where $\underline{L}_P$ is a fiber of $\pi : \underline{L} \rightarrow X$ at $P \in X$. We define the multiplicative height function associated to $\mathcal L = (L, \|\cdot\|_v)$ by
\[
H_{\mathcal L}(P) = \prod_{v \in M_k} \|\ell\|_v^{-1}.
\]
Note that this does not depend on the choice of $\ell$ because of the product formula (\ref{eqn:product}). We also define the logarithmic height function by
\[
h_{\mathcal L}(P) = \log H_{\mathcal L}(P).
\]
\end{defi}

\begin{rema}
There is another construction of height functions using the framework of Weil height machines (\cite[Theorem B.3.2]{Silver00}). One can show that two constructions are equivalent in the sense that two height functions associated to the same line bundle are equal up to a bounded function. (See \cite[Theorem B.10.7]{Silver00}) Also note that the construction of height functions in this paper uses normalizations which differ from ones in \cite{Silver00}. This is because in this paper we only consider height functions defined on the set of rational points while \cite{Silver00} considers the height functions defined over the set of algebraic points. Thus in \cite{Silver00}, one needs to normalize each height function by the degree of the definition field of an algebraic point.
\end{rema}

\begin{rema}
In the later discussions, we frequently omit the discussion of metrics and we consider a height function $h_L$ associated to a line bundle $L$. In this situation, we implicitly make a choice of an adelic metric, but we will not make this dependence explicit as this does not matter for our discussion.
\end{rema}

There are two important properties of height functions we frequently use in this paper:

\begin{theo}{\cite[Theorem B.3.2]{Silver00}}
Let $X$ be a projective variety defined over $k$ and $L$ be an invertible sheaf on $X$.
\begin{enumerate}
\item (Positivity) Let $B$ be the stable base locus of $L$. Then we have
\[
h_L(P) \geq O(1) 
\]
for any $P \in (X \setminus B)(k)$.
\item (the Northcott property)
Suppose that $L$ is ample. Then for any $T > 0$ the set
\[
\{P \in X(k) \in H_L(P) \leq T\}
\]
is finite.
\end{enumerate}
\end{theo}

Finally we recall the construction of local height functions:

\begin{defi}
Let $X$ be a projective variety defined over $k$ and $L$ be a Cartier divisor on $X$. For a place $v \in M_k$ we fix a $v$-adic metric $\|\cdot\|_v$ for $L$, i.e., a $v$-adic metric on the analytic line bundle $\pi : \underline{\mathcal O(L)}(k_v) \rightarrow X(k_v)$. Let $D$ be an effective divisor linearly equivalent to $L$. Let $\mathsf s_D$ be a $k$-section associated to $D$. Then the multiplicative local height function associated to $D$ is given by
\[
H_{D, v}(P) = \|\mathsf s_D(P)\|_v^{-1}
\]
for any $P\in (X \setminus D)(k_v)$. We also define the logarithmic local height function by
\[
h_{D, v}(P) := \log H_{D, v}(P).
\]
\end{defi}
Suppose we fix an adelic metrized line bundle $\mathcal L$.
Then the height function is the Euler product of local height functions, i.e., we have
\[
H_{\mathcal L}(P) = \prod_{v \in M_k} H_{D, v}(P)
\]
for any $P \in (X \setminus D)(k)$.

\section{The Fujita invariant in Manin's conjecture}
\label{sect:fujita}

Here we assume that our ground field $k$ is a field of characteristic zero,
but not necessarily algebraically closed.
Recently the geometric study of Fujita invariants has been conducted in a series of papers \cite{HTT15}, \cite{LTT14}, \cite{HJ16}, \cite{LTDuke}, \cite{LT17}, \cite{Sen17}, \cite{LST18}, \cite{LT18}, \cite{LT18-b}. We recall its definition here.
\begin{defi}
Let $X$ be a smooth projective variety defined over $k$. Let $L$ be a big and nef $\mathbb Q$-divisor on $X$. We define the {\it Fujita invariant} (or $a$-invariant) by
\[
a(X, L) = \inf\{ t \in \mathbb R \mid K_X + tL \in \overline{\mathrm{Eff}}^1(X)\},
\]
where $\overline{\mathrm{Eff}}^1(X)$ is the cone of pseudo-effective divisors on $X$. By \cite{BDPP} $a(X, L)>0$ if and only if $X$ is geometrically uniruled. When $L$ is not big, we simply set $a(X, L) = +\infty$. When $X$ is singular, we take a resolution $\beta: X' \rightarrow X$ and we define the Fujita invariant by
\[
a(X, L) := a(X', \beta^*L).
\]
This is well-defined because the Fujita invariant is a birational invariant (\cite[Proposition 2.7]{HTT15}).
\end{defi}

This invariant plays a central role in Manin's conjecture. For example, one can predict the exceptional set of Manin's conjecture by studying this invariant and the following result is a consequence of Birkar's celebrated papers \cite{birkar16} and \cite{birkar16b}:

\begin{theo}[\cite{LTT14}, \cite{HJ16}, \cite{LT17}]
\label{theo:closed}
Assume that our ground field is algebraically closed.
Let $X$ be a smooth projective uniruled variety and let $L$ be a big and nef $\mathbb Q$-divisor on $X$. Let $V$ be the union of subvarieties $Y$ with $a(Y, L)>a(X, L)$. Then $V$ is a proper closed subset of $X$.
\end{theo}

For computations of this exceptional set $V$ for some examples, see \cite{LTT14} and \cite{LT18}. 



\section{The invariant $\delta(X, H)$}
\label{sec:delta}

Here we assume that our ground field $k$ is an algebraically closed field of characteristic $0$.
Let $X$ be a normal projective variety and $H$ be a big $\mathbb Q$-Cartier divisor on $X$. 
We consider $W = X \times X$ and denote each projection by $\pi_i : W \rightarrow X_i$.
Let $\alpha : W' \rightarrow W$ be the blow up of the diagonal and we denote its exceptional divisor by $E$.
For any $\mathbb Q$-Cartier divisor $L$ on $X$ we denote $\alpha^*\pi_1^*L + \alpha^*\pi_2^*L$ by $L[2]$.

\begin{defi}
Let $X, W', E$ as above. We define the following invariant 
\[
\delta(X, H) = \inf\left\{s\in \mathbb R \middle| \begin{matrix}\text{ for any component $V \subset \mathrm{SB}(sH[2]- E)$ not contained in $E$,} \\
\text{one of $\pi_i\circ \alpha|_V$ is not dominant to $X_i$. }
\end{matrix}\right\},
\] 
where $\mathrm{SB}(sH[2]- E)$ is the stable base locus of a $\mathbb R$-divisor $sH[2]- E$.
\end{defi}
\begin{rema}
\label{rema:elementary}
It follows from the definition that when the rational map $\Phi_{|H|}$ associated to $|H|$ is birational, we have $\delta(X, H) \leq 1$. Indeed, let $Z$ be a closed subset such that on $X \setminus Z$, $\Phi_{|H|}$ is well-defined and an isomorphism onto the image. Then one can conclude that
\[
\mathrm{Bs}(H[2] - E) \subset E\cup \pi_1^{-1}(Z) \cup \pi_2^{-1}(Z),
\]  
where $\mathrm{Bs}(H[2] - E)$ is the base locus of $H[2] - E$.
Thus our assertion follows.
Also it follows from the definition that $\delta(X, H)H[2]-E$ is pseudo-effective.
\end{rema}

\begin{lemm}
\label{lemm:coveringcurves}
Let $s_0$ be a positive real number.
Let $F_P$ be a fiber of the first projection $\pi_1\circ \alpha : W' \rightarrow X_1$ at $P \in X_1$.
Suppose that there exists a family of irreducible curves $C_t \subset F_{P(t)}$ such that (i) $P(t)$ covers a Zariski open subset of $X_1$ as $t$ varies (ii) the image of $C_t$ in $X_2$ is an irreducible curve containing $P(t)$ (iii) $(s_0H[2] -E).C_t =0$. Then we have $\delta(X, H)\geq s_0$.
\end{lemm}
\begin{proof}
Let $V = \overline{\cup C_t}$ in $W'$ which is irreducible. Then $V$ is not contained in $E$ and each projection $\pi_i|_{V}$ is dominant. Let $s < s_0$. Then since $(sH[2] -E).C_t<0$, $\mathrm{SB}(sH[2]- E)$ contains $V$. Thus we must have $s \leq \delta(X, H)$. Since this is true for any $s < s_0$, our assertion follows.
\end{proof}

\begin{exam}
Let $X = \mathbb P^n$ and $H$ be the hyperplane class. Then $\delta(X, H) = 1$.
Indeed, it follows from Remark~\ref{rema:elementary} that $\delta(X, H) \leq 1$. On the other hand, let $F_1$ be a general fiber of the first projection $\pi_1\circ \alpha : W' \rightarrow X_1 = \mathbb P^n$ at $P \in X_1$ and $\ell$ be the strict transform of a line passing through $P$ in $F_1$. Then we have $(H[2]- E).\ell = 0$. Thus our assertion follows from Lemma~\ref{lemm:coveringcurves}.

\end{exam}

\begin{exam}
Let $X \subset \mathbb P^n$ be a normal projective variety and $H$ be the hyperplane class. Suppose that $X$ is covered by lines. Then the same proof of the above example shows that $\delta(X, H) = 1$.
\end{exam}

Next we show that the invariant $\delta(X, H)$ is a birational invariant.

\begin{lemm}
\label{lemm:birational}
Let $X$ be a normal projective variety and $H$ be a big $\mathbb Q$-Cartier divisor on $X$. 
Let $\beta: X' \rightarrow X$ be a birational morphism between normal projective varieties. Then we have
\[
\delta(X', \beta^*H) = \delta(X, H).
\]
\end{lemm}
\begin{proof}
Let $W_X$ be the blow-up of $X\times X$ along the diagonal and $W_{X'}$ be the blow-up of $X'\times X'$ along the diagonal. We denote their exceptional divisors by $E_{X}$ and $E_{X'}$ respectively. Then we have a birational map
\[
\phi: W_{X'} \dashrightarrow W_X
\]
which is a birational contraction and the indeterminacy of this map is not dominant to both $X_i'$.
Also for a component $V$ of the non-isomorphic loci of this map such that $V$ is not contained in $E_{X'}$, one of projections is not dominant.

Fix $\epsilon > 0$. Suppose that the stable locus of $(\delta(X, H) + \epsilon)\beta^*H[2] - E_{X'}$ contains a subvariety $Y \subset W_{X'}$ such that $Y \not\subset E_{X'}$ and $Y$ maps dominantly to both $X'_i$. By the definition, $(\delta(X, H) + \epsilon)H[2] - E_{X}$ does not contain $\phi(Y)$ in the stable locus so that there exists $0\leq D \sim_{\mathbb R} (\delta(X, H) + \epsilon)H[2] - E_{X}$ such that $\phi(Y) \not\subset \mathrm{Supp}(D)$.
Then we have $\phi^*D \sim_{\mathbb R} (\delta(X, H) + \epsilon)\beta^*H[2] - E_{X'}$ because $\phi^*E_{X} = E_{X'}$. Furthermore we have $Y \not\subset \mathrm{Supp}(\phi^*D)$. This contradicts with our assumption. Thus we conclude
\[
\delta(X', \beta^*H) \leq \delta(X, H).
\]

Suppose that the stable locus of $(\delta(X', H) + \epsilon)H[2] - E_{X}$ contains a subvariety $Y \subset W_{X}$ such that $Y \not\subset E_{X}$ and $Y$ maps dominantly to both $X_i$. We take the strict transform $Y' \subset W_{X'}$ of $Y$. By the definition, $(\delta(X', H) + \epsilon)\beta^*H[2] - E_{X'}$ does not contain $Y'$ in the stable locus so that there exists $0\leq D \sim_{\mathbb R} (\delta(X', H) + \epsilon)\beta^*H[2] - E_{X'}$ such that $Y' \not\subset \mathrm{Supp}(D)$.
Then we have $\phi_*D \sim_{\mathbb R} (\delta(X', H) + \epsilon)H[2] - E_{X}$. Furthermore we have $Y \not\subset \mathrm{Supp}(\phi_*D)$. This contradicts with our assumption. Thus we conclude
\[
\delta(X', \beta^*H) \geq \delta(X, H).
\]
Thus our assertion follows.
\end{proof}

Here is a relation between $\delta(X, H)$ and $a(X, H)$.

\begin{prop}
\label{prop:relation}
Let $X$ be a smooth weak Fano variety, i.e., $-K_X$ is big and nef, and $H$ be a big and nef divisor on $X$.
Then we have
\[
\delta(X, H) \leq a(X, H)\delta(X, -K_X).
\]
\end{prop}
\begin{proof}
We write $a(X, H)H + K_X \sim_{\mathbb Q} D \geq 0$. Fix $\epsilon > 0$. Then we have
\[
a(X, H)(\delta(X, -K_X)+\epsilon)H[2]-E \sim_{\mathbb Q} -(\delta(X, -K_X)+\epsilon)K_X[2] + (\delta(X, -K_X)+\epsilon)D[2] -E.
\]
Thus we see that the stable locus of $|a(X, H)(\delta(X, -K_X)+\epsilon)H[2]-E|$ does not contain any dominant component possibly other than subvarieties in $E$. Thus our assertion follows.
\end{proof}

Next we consider the $s$-invariants and its relation to the $\delta$-invariants:
\begin{defi}
Let $X$ be a smooth projective variety and $H$ be an ample divisor on $X$.
Let $W$ be the blow-up of $X\times X$ along the diagonal and we denote its exceptional divisor by $E$. The $s$-invariant of $H$ is defined by 
\[
s(X, H) = \inf \{s \in \mathbb R | \text{ $sH[2] - E$ is nef}\}.
\]
This is a positive real number in general. See \cite[Section 5.4]{LazI} for many properties of this invariant.
\end{defi}

\begin{prop}
\label{prop:s-invariant}
Let $X$ be a smooth projective variety and $H$ be an ample divisor on $X$.
Then we have
\[
\delta(X, H)\leq s(X, H).
\]
\end{prop}

\begin{proof}
For every $\epsilon >0$, $(s(X, H) + \epsilon)H[2] -E$ is ample so its stable base locus is empty. Thus our assertion follows.
\end{proof}

\subsection{Del Pezzo surfaces}
\label{subsec:delpezzo}
Next we discuss del Pezzo surfaces.
Let $S$ be a smooth del Pezzo surface. 
We consider $W = S \times S$ and we denote each projection by $\pi_i : W \rightarrow S_i$. Let $\alpha: W' \rightarrow W$ be the blow up of the diagonal and we denote its exceptional divisor by $E$. 
First we record a lower bound for the $\delta$-invariant:
\begin{lemm}
\label{lemm:lowerboundfordelta}
Let $S$ be a smooth del Pezzo surface.
Then we have
\[
\delta(S, -K_S)\geq \frac{1}{\epsilon(-K_S, P)}
\]
for any general point $P\in X$ where $\epsilon(-K_S, P)$ is the Seshadri constant of $-K_S$ at $P$.
\end{lemm}
\begin{proof}
The Seshadri constant for the anticanonical divisor on a smooth del Pezzo surface is computed in \cite{Bro06}. According to this paper, $\epsilon(-K_S, P)$ is constant for a general point $P \in S$ and for such a $P$ we have
\[
\epsilon(-K_S, P) = \min_{P \in C \subset S} \frac{-K_S.C}{\mathrm{mult}_P(C)}.
\]
Moreover, curves achieving the minimum are completely described for del Pezzo surfaces of degree $\geq 2$ in \cite{Bro06} and they are members of one family from the Hilbert scheme. For a del Pezzo surface of degree $1$, this minimum is achieved by members of the anticanonical system.

Let $P \in S_1$ be a general point and let $C_t$ be the strict transform of a curve in $\{P\} \times S_2$ achieving the minimum $\epsilon(-K_S, P)$. Then we have
\[
(-K_S[2] - \epsilon(-K_S, P)E).C_t = -K_X.C_t - \epsilon(-K_S, P)\mathrm{mult}_P(C_t) = 0.
\]
Thus our assertion follows from Lemma~\ref{lemm:coveringcurves}.
\end{proof}

Now we compute $\delta(S, -K_S)$ for a del Pezzo surface $S$. 
\begin{prop}
\label{prop:degree4}
Let $S$ be a del Pezzo surface of degree $d$ where $4 \leq d \leq 8$. Then we have $\delta(S, -K_S) = \frac{1}{2}$.
\end{prop} 

\begin{proof}
We only discuss the case of degree $4$ del Pezzo surfaces. Other cases are easier.

Suppose that $S$ is a del Pezzo surface of degree $4$.
Let $F_1$ be a $-K_S$-conic on $S$ and $F_2$ be another $-K_S$-conic on $S$ such that $-K_S \sim F_1 + F_2$. Indeed, one may find such a pair of $-K_S$-conics in the following way: let $\phi : S \rightarrow \mathbb P^2$ be a blow down to $\mathbb P^2$ and we may assume that $\phi$ is the blow up at $P_1, \cdots, P_5 \in \mathbb P^2$. Then one can find a general plane conic $C_1$ and a general line $C_2$ such that $C_1$ contains $P_1, \cdots, P_4$ and $C_2$ contains $P_5$. Then their strict transforms satisfy the desired property. 

Now the linear system $|F_i|$ defines a conic fibration $p_i : S \rightarrow \mathbb P^1$. It induces a morphism $p_i[2] : W' \rightarrow \mathbb P^1 \times \mathbb P^1$. Let $\Delta_{\mathbb P^1}$ be the diagonal of $\mathbb P^1 \times \mathbb P^1$. Then $F_i[2] -E$ is linearly equivalent to a unique effective divisor $p_i[2]^*\Delta_{\mathbb P^1}-E$. We denote it by $\Delta_{F_i}$. Let $D_1$ be a third conic such that $D_1.F _1 = D_1 .F_2 = 1$. Indeed, one may take $D_1$ as the strict transform of a conic passing through $P_1, P_2, P_3, P_5$. Let $D_2$ be the class of conics such that $-K_S \sim D_1 + D_2$. Then we have $D_2.F_1 = D_2.F_2 = 1$. We also consider $\Delta_{D_i}$. Then we have
\[
\mathrm{SB}(-K_S[2]- 2E) \subset (\Delta_{F_1}\cup \Delta_{F_2}) \cap (\Delta_{D_1}\cup \Delta_{D_2}) 
\]
but the only possible dominant component of $\Delta_{F_i} \cap \Delta_{D_i}$ is contained in $E$.
Thus we conclude that $\delta(X, H) \leq 1/2$. The opposite inequality follows from Lemma~\ref{lemm:lowerboundfordelta} and \cite{Bro06}.


\end{proof}

\begin{prop}
\label{prop:cubic}
Let $S$ be a del Pezzo surface of degree $3$. Then we have $\delta(S, -K_S) = \frac{2}{3}$.
\end{prop}
\begin{proof}
Let $F_1$ be a $-K_S$-conic. One can find a $(-1)$-curve $E$ such that $-K_S \sim F_1 + E$. Indeed, let $\phi : S \rightarrow \mathbb P^2$ be the blow up at $P_1, \cdots, P_6 \in \mathbb P^2$. We let $F_1$ to be the strict transform of a general conic passing through $P_1, \cdots, P_4$ and $E$ be the strict transform of a line passing through $P_5, P_6$. Then they satisfy the desired property. Now let $D = - 2K_S - F_1 \sim -K_S + E$ then $D$ is the pullback of the anticanonical class from a degree $4$ del Pezzo surface. The upshot is that we have
\[
-2K_S[2] - 3E = D[2] - 2E + F_1[2]-E.
\]
Thus it follows from the proof of Proposition~\ref{prop:degree4} that the stable locus of $D[2]-2E$ minus $E$ is not dominant. Thus the stable locus of $-2K_S[2] - 3E$ is contained in $\Delta_{F_1}$. By considering another conic and applying the same discussion, we conclude that $\delta(S,-K_S) \leq 2/3$. The opposite inequality follows from Lemma~\ref{lemm:lowerboundfordelta} and  \cite{Bro06}.

\end{proof}

\begin{prop}
\label{prop:degree2}
Let $S$ be a del Pezzo surface of degree $2$. Then we have $\delta(S, -K_S) = 1$.
\end{prop}

\begin{proof}
We may write $-K_S \sim E_1 + E_2$ where $E_i$ is a $(-1)$-curve.
Indeed, let $\phi : S \rightarrow \mathbb P^2$ be the blow up at $P_1, \cdots, P_7$. Then we may define $E_1$ as the strict transform of a conic passing through $P_1, \cdots, P_5$ and $E_2$ be the strict transform of a line passing through $P_6, P_7$.
Let $f_i: S \rightarrow S_i$ be the blow down of $E_i$ to a cubic surface.
Then $-3K_S$ can be expressed as 
\[
-3K_S \sim -f_1^*K_{S_1}- f_2^*K_{S_2}.
\]
Thus arguing as Proposition~\ref{prop:cubic} we prove that the stable locus of $-K_S[2]-E$ does not contain any dominant component other than $E$. This shows that $\delta(S, -K_S) \leq 1$.

On the other hand, let $\phi : S \rightarrow \mathbb P^2$ be the anticanonical double cover.
We denote the involution associated to $\phi$ by $\iota$
and we consider the image $S^\iota$ of the following map
\[
S \rightarrow S\times S, P \mapsto (P, \iota(P)).
\]
Then one can show that for any curve $C$ in $S^\iota$ and any $\epsilon > 0$ we have
\[
(-K_S[2]-(1 + \epsilon)E).C <0
\]
Thus $C$ is contained in the stable locus of $-K_S[2]-(1 + \epsilon)E$, proving the claim.

\end{proof}

\begin{prop}
\label{prop:degree1}
Let $S$ be a del Pezzo surface of degree $1$. Then we have $3/2 \leq \delta(S, -K_S) \leq 2$.
\end{prop}

\begin{proof}
Let $\phi : S \rightarrow Q \subset \mathbb P^3$ be the double cover associated to $|-2K_S|$.
Let $E_1$ be a $(-1)$-curve on $S$. Then $\phi|_{E_1} : E_1 \rightarrow \phi(E_1)$ is 1:1 and its pullback consists of two $(-1)$-curves including $E_1$. Thus we may write as $-2K_S \sim E_1 + E_2$.

Let $f_i : S \rightarrow S_i$ be the blow down of $E_i$ to a degree $2$ del Pezzo surface.
Then $-4K_S$ can be expressed as 
\[
-4K_S \sim -f_1^*K_{S_1}  -f_2^*K_{S_2}  
\]
Thus by Proposition~\ref{prop:degree2}, one may conclude that $\delta(S, -K_S) \leq 2$.

Another inequality follows from the discussion of Proposition~\ref{prop:degree2} using the double cover $\phi : S \rightarrow Q \subset \mathbb P^3$. 



\end{proof}

\begin{rema}
In the proof of Proposition~\ref{prop:degree4}, \ref{prop:cubic}, \ref{prop:degree2} we show that for any component $V \subset \mathrm{SB}(-\delta(S, -K_S)K_S[2]- E)$ not contained in $E$, the projection $\pi_i\circ \alpha|_V$ is not dominant. In particular we conclude that
\[
\delta(S, -K_S) = \inf\left\{s\in \mathbb R \middle| \begin{matrix}\text{ for any component $V \subset \mathrm{SB}(-sK_S[2] - E)$ not contained in $E$,} \\
\text{one of $\pi_i\circ \alpha|_V$ is not dominant to $S_i$. }
\end{matrix}\right\},
\] 
is the minimum.
\end{rema}

\section{Repulsion principle for projective varieties}
\label{sec:repulsion}

Assuming Vojta's conjecture and the Non-Vanishing conjecture in the minimal model program, McKinnon shows a Repulsion principle for varieties of non-negative Kodaira dimension in \cite{Mck11}. In this paper we develop a weaker Repulsion principle for projective varieties in general. We introduce some notations. We refer readers to \cite{Sil87} for the definitions and their basic properties.

Let $k$ be a number field. Suppose that we have a projective variety $X$ defined over $k$ and a big $\mathbb Q$-divisor $L$ on $X$. Let $D$ be a closed subscheme on $X$. 
Let $h_{D, v}$ be a local height function for $D$ with respect to $v$. Note that in this paper, we use unnormalized heights, i.e., we do not normalize heights by the degree of $k$.
Let $\Delta$ be the diagonal of $X \times X$. We define the $v$-adic distant function by
\[
h_{\Delta, v}(P, Q) = - \log \mathrm{dist}_v(P, Q).
\]
See \cite{Sil87} for basic properties of this function. 

Let $X$ be a normal projective variety defined over a number field $k$ and $L$ be a big $\mathbb Q$-Cartier divisor on $X$. We set
\[
\delta(X, L) = \delta(\overline{X}, \overline{L})
\]
where $\overline{X}, \overline{L}$ are the base change of $X$, $L$ to an algebraic closure.
Here is our main theorem:

\begin{theo}[the Repulsion principle]
\label{theo:repulsion}
Let $X$ be a normal projective variety defined over a number field $k$. Let $v$ be a place of $k$. Let $A$ be a big Cartier divisor on $X$. Then for any $\epsilon > 0$ there exists a constant $C= C_\epsilon > 0$ and a non-empty Zariski open subset $U = U(\epsilon) \subset X$ such that we have
\[
\mathrm{dist}_v(P, Q) > C (H_A(P)H_A(Q))^{-(\delta(X, A) + \epsilon)},
\]
for any $P, Q \in U(k)$ with $P \neq Q$.
\end{theo}

\begin{proof}
We let $W = X \times X$ with projections $\pi_i: W \rightarrow X_i$ and we let $L = \pi_1^*A + \pi_2^*A$. We denote the blow up of the diagonal by $\alpha : W'\rightarrow W$ and its exceptional divisor by $E$.
Fix $\epsilon >  0$. Let $\overline{B} \subset \overline{W}'$ be the stable locus $\mathrm{SB}((\delta(X, A) + \epsilon)\alpha^*L - E)$ where $\overline{W}'$ is the base change to an algebraic closure. One can express $\overline{B}$ as the intersection of supports of finitely many effective $\mathbb R$-divisors which are $\mathbb R$-linearly equivalent to $(\delta(X, A) + \epsilon)\alpha^*L - E$. After taking some finite extension $k'$ of $k$, we may assume that these divisors are defined over $k'$ so does $\overline{B}$. We denote the union of the Galois orbits of $\overline{B}$ by $B'$.
Then it is a property of height functions that for any $(P, Q) \in W'(k)\setminus B'(k)$, we have
\[
0\leq h_{((\delta(X, A) + \epsilon)\alpha^*L - E)}(P,Q) + O(1) 
\]
From this, we may conclude that 
\[
h_{E, v}(P, Q) \leq h_{E}(P, Q) + O(1) \leq h_{((\delta(X, A) + \epsilon)\alpha^*L )}(P,Q) + O(1).
\]
Let $V \subset B'$ be a component not contained in $E$. Then one of projections $\pi_i\circ\alpha|_V$ is not dominant, and we denote its image by $F_V$. Now we define $U$ by $X \setminus \cup_{V}F_V$.
Our assertion follows for this $U$.
\end{proof}

\begin{rema}
\label{rema:minimum}
Note that $\delta(X, A)$ is defined as 
\[
\delta(X, A) = \inf\left\{s\in \mathbb R \middle| \begin{matrix}\text{ for any component $V \subset \mathrm{SB}(s\alpha^*L - E)$ not contained in $E$,} \\
\text{one of $\pi_i\circ \alpha|_V$ is not dominant to $X_i$. }
\end{matrix}\right\}.
\] 
If this is the minimum, then in the above proof, one does not need to introduce $\epsilon > 0$.
\end{rema}

\begin{rema}
\label{rema:s-inv}
When $A$ is ample, we may replace $\delta(X, A)$ by $s(X, A)$ in the above theorem. In this situation, one can take our exceptional set to be empty because of the emptiness of the base locus in the proof of Proposition~\ref{prop:s-invariant}.
\end{rema}

\section{Counting problems: general cases}
\label{sect:general}

In this section, we discuss some applications of Theorem~\ref{theo:repulsion} to the counting problems of rational points on algebraic varieties.

\subsection{Local Tamagawa measures}

Here we record some auxiliary results for local Tamagawa measures.
Let $X$ be a smooth projective variety defined over a number field $k$.
Let $v$ be a place of $k$. We fix a $v$-adic metrization on $\mathcal O(K_X)$ and it induces the Tamagawa measure $\tau_{X, v}$ on $X(k_v)$. We refer readers to \cite[Section 2.1.8]{volume} for its definition.

\begin{lemm}
\label{lemm: smallball}
Let $n = \dim X$.
There exists $C>0$ such that for sufficiently small $T$ and $P\in X(k_v)$, we have
\[
CT^n < \tau_{X, v}(\{Q \in X(k_v)\mid \mathrm{dist}_v(P, Q)<T\} ).
\]
\end{lemm}

\begin{proof}
Let $Y = X \times X$. We take a finite open cover $\{U_i\}$ of $Y$ such that on $U_i$, $\Delta$ is the scheme-theoretic intersection of $D_{i, 1}, \cdots, D_{i, n}$ where $D_i = \sum D_{i, j}$ is a strict normal crossings divisor on $U_i$. On $U_i(k_v)$, there exists $C > 0$ such that
\[
\mathrm{dist}_v(P, Q) < C \max_j \{H_{D_{i, j}}^{-1}(P, Q)\}
\]
for all $(P, Q) \in U_i(k_v)$.
For each $P \in X_i(k_v)$, there exists a $v$-adic open neighborhood $V_P \subset X(k_v)$ such that $\overline{V}_P \times \overline{V}_P \subset U_i(k_v)$ for some $i$ and $D_{i, j}$ induces local coordinates $x_{i, j}$ on $V_P$. Since $X(k_v)$ is compact, finitely many $V_P$ covers $X(k_v)$. We denote them by $V_l$. For each $l$ let $$\omega_l = \mathrm d x_{l, 1} \wedge \cdots \wedge  \mathrm d x_{l, l}.$$
Then on $V_l$ we have a uniform upper bound $C' > 0$ such that
\[
\|\omega_l\|_v < C'.
\]
Also let $d_l(P) = \min \{\mathrm{dist}_v(P, Q) \mid Q \in  V_l^c\}$
and we define $d(P) = \max_l\{d_l(P) \}$. Then $d(P) >0$ for any $P \in X(k_v)$ so there is the minimum $d_m = \min\{d(P)\} > 0$. 
Now by the definition of the Tamagawa measure, for $0< T < d_m$, we have
\[
\tau_{X, v}(\{Q \in X(k_v)\mid \mathrm{dist}_v(P, Q)<T\} ) > \int_{\{\max\{|x_{l,i}-x_{l,i}(P)|_v\} < C^{-1}T\}}\|\omega_l\|_v^{-1} |\omega_l|
\]
Thus our assertion follows.

\end{proof}

The local Tamagawa number is defined by
\[
\tau_v(X) = \tau_{X, v}(X(k_v)).
\]

\subsection{General estimates}

Let $X$ be a projective variety defined over a number field $k$ and $L$ be a big $\mathbb Q$-divisor on $X$. We fix an adelic metrization on $\mathcal O(L)$ and consider the induced height:
\[
H_L : X(F) \rightarrow \mathbb R_{>0}.
\]
For each Zariski open subset $U \subset X$ we define the counting function:
\[
N(U, L, T) = \#\{ P \in U(k)\mid H_L(P) \leq T\}.
\]
Here is a general result using Repulsion principle:
\begin{theo}
\label{theo:general}
Let $X$ be a normal projective variety of dimension $n$ defined over a number field $k$ and $L$ be a big $\mathbb Q$-Cartier divisor on $X$. We fix an adelic metrization on $\mathcal O(L)$. Then for any $\epsilon > 0$ there exists a non-empty Zariski open subset $U = U(\epsilon) \subset X$ such that  we have
\[
N(U, L, T) = O(T^{2n\delta(X, L) + \epsilon}).
\]
\end{theo}

\begin{proof}
We may assume that $X$ is smooth after applying a resolution.
This does not affect the invariant $\delta(X, L)$ because of Lemma~\ref{lemm:birational}.
Let $v$ be a place of $k$.
By Theorem~\ref{theo:repulsion}, for any $\epsilon > 0$ there exists a non-empty Zariski open subset $U(\epsilon)\subset X$ such that  there exists $C = C_\epsilon >0$ such that 
\[
\mathrm{dist}_v(P, Q) > C(H_L(P)H_L(Q))^{-(\delta(X, L) + \epsilon)},
\]
for any $P, Q \in U(k)$ with $P \neq Q$.
We define
\[
A_T = \{P \in U(k) \mid H_L(P) \leq T\}.
\]
For $P \in A_T$, we define the $v$-adic ball by
\[
B_T(P) = \{R \in U(k_v) \mid \mathrm{dist}_v(P, R)< \frac{1}{2}CT^{-2(\delta(X, L) + \epsilon)}\}
\]
Then $\cup_{P \in A_T} B_T(P)$ is disjoint because of the triangle inequality. Hence we have
\[
\tau_v(X) > \sum_{P \in A_T} \tau_{X, v}(B_T(P) ) \gg N(U, L, T)T^{-2n(\delta(X, L) + \epsilon)}
\] 
by Lemma~\ref{lemm: smallball}. Thus our assertion follows.
\end{proof}
\begin{rema}
In the case that $\delta(X, L)$ is the minimum, then one does not need to introduce $\epsilon$ in the above theorem because of Remark~\ref{rema:minimum}.
\end{rema}
\begin{rema}
\label{rema:s}
In the above theorem, assuming $L$ is ample and $X$ is smooth we may replace $\delta(X, L)$ by $s(X, L)$ because of Remark~\ref{rema:s-inv}. In this case, one can take $U =X$.
\end{rema}

In view of Manin's conjecture, we expect the following is true:
\begin{conj}
Let $X$ be a geometrically rationally connected smooth projective variety of dimension $n$ and $L$ be a big and nef $\mathbb Q$-divisor on $X$. Then we have
\[
a(X, L) \leq 2n \delta(X, L).
\]
\end{conj}

\section{K3 surfaces and Enriques surfaces}
\label{sec:K3}
In this section we discuss applications of Theorem~\ref{theo:general} to surfaces of Kodaira dimension $0$.
Let $S$ be a K3 surface or an Enriques surface with a polarization $H$ of degree $2d$. In this section, we obtain an upper bound for $s(X, H)$ using \cite{BM14} and \cite{Nuer}. Let $W$ be the blow-up of $S \times S$ along the diagonal and we denote the exceptional divisor by $E$. We also consider the Hilbert Scheme of two points on $S$, i.e., $\mathrm{Hilb}^{[2]}(S)$. The variety $\mathrm{Hilb}^{[2]}(S)$ comes with the divisor $H(2)$ induced by $H$ and a divisor class $B$ such that $2B$ is the class of the exceptional divisor of the Hilbert-Chow morphism. The variety $W$ admits a degree $2$ finite morphism $f : W \rightarrow \mathrm{Hilb}^{[2]}(S)$ and we have
\[
f^*H(2) = H[2], \quad f^*B = E
\]
We then have
\[
\text{$sH[2] - E$ is nef} \iff \text{$sH(2) - B$ is nef},
\]
because of $f^*(sH(2) - B) = sH[2] - E$.
Thus one needs to study the nef cone of $\mathrm{Hilb}^{[2]}(S)$
and this is studied in \cite{HT01}, \cite{HT09}, and \cite{BM14} for K3 surfaces.
We use results from \cite{BM14} for the nef cone of $\mathrm{Hilb}^{[2]}(S)$.
Here is the theorem:
\begin{theo}{\cite{BM14}}
Let $S$ be a K3 surface with a polarization $H$ of degree $2d$ such that $\mathrm{Pic}(S) = \mathbb ZH$. Then we have
\[
s(S, H) \leq \sqrt{\frac{4}{d} + \frac{5}{d^2}}.
\]
\end{theo}

\begin{proof}
We recall a result on the nef boundary of $sH(2) - B$ based on properties of certain Pell's equation.
First we consider
\[
X^2 - 4dY^2 = 5.
\]
Suppose that there is a non-trivial solution  $(x_1, y_1)$ with $x_1 > 0$ minimal and $y_1 > 0$ even. Then it follows from \cite[Lemma 13.3]{BM14} that
\[
s(S, H) = \frac{x_1}{dy_1}\leq \sqrt{\frac{4}{d} + \frac{5}{d^2}}.
\] 
Next suppose that there is no non-trivial solution to the above Pell's equation. Then for $\mathrm{Hilb}^{[2]}(S)$, the nef cone and the movable cone coincides by \cite[Lemma 13.3]{BM14}. Suppose that $d$ is a square. Then it follows from \cite[Proposition 13.1]{BM14} that
\[
s(S, H) = \frac{1}{\sqrt{d}}.
\]
Next suppose that $d$ is not a square. We consider the following Pell's equation:
\[
X^2 - dY^2 = 1.
\]
This has a solution. Let $x_1, y_1 > 0$ be the solution with $x_1$ minimal. Then by \cite[Proposition 13.1]{BM14}, we have
\[
s(S, H) = \frac{x_1}{dy_1}\leq \sqrt{\frac{1}{d} + \frac{1}{d^2}}.
\]
Thus our assertion follows.
\end{proof}

Now Theorem~\ref{theo:K3} follows from the above theorem and Remark~\ref{rema:s}.
We also obtain bounds for Enriques surfaces:

\begin{theo}[\cite{Nuer}]
\label{theo:enriques}
Let $Y$ be an unnodal Enriques surface, i.e., $Y$ contains no curve of negative self-intersection.
Let $H$ be a $k$-very ample divisor. Then we have
\[
s(Y, H)\leq \frac{2}{k+2}
\]
\end{theo}

\begin{proof}
It follows from \cite[Theorem 12.3]{Nuer} that 
\[
s(Y, H) = \frac{2}{\phi(H)},
\]
where $\phi(H)$ is the Cossec-Dolgachev function. Then it follows from \cite[Theorem 2.4]{Szem} that $\phi(H) \geq k+2$. Thus our assertion follows.
\end{proof}

For a necessary and sufficient condition for $k$-very ampleness, see \cite[Proposition 2.3]{Szem}

\section{Manin type upper bounds for Fano conic bundles}
\label{sect:countingonconicbundle}

In this section, we study the counting problems of rational points on conic bundles.

\begin{defi}
Let $f: X \rightarrow S$ be a flat projective morphism between smooth projective varieties. The fibration $f$ is a conic bundle if there exist a rank $3$ vector bundle $\mathcal E$ on $S$ and an embedding $X \rightarrow \mathbb P_S(\mathcal E)$ over $S$ such that every fiber $X_s$ is isomorphic to a conic in $\mathbb P^2_s$ where $X_s$ and $\mathbb P^2_s$ are fibers of $X$ and $\mathbb P(\mathcal E)$ at $s \in S$.

Suppose that $X$ is $3$-dimensional. Then if every fiber $X_s$ is isomorphic to a conic in $\mathbb P^2$, then $f_*\omega_X^{-1}$ is a rank $3$ vector bundle on $S$ and a natural map $X \rightarrow \mathbb P_S(f_*\omega_X^{-1})$ is an embedding by \cite[Proposition 6.2]{MM83}. In this way $f :X \rightarrow S$ is a conic bundle in the above sense.
\end{defi}

\begin{lemm}
\label{lemm:generalconicbundle}
Let $f: X \rightarrow S$ be a conic bundle and $H$ be a big $\mathbb Q$-divisor on $X$.
Suppose that for a fiber $X_s$ of $f$, we have $H. X_s = 2$, i.e., $X_s$ is a $H$-conic. Then we have
\[
\delta(X, H) \geq \frac{1}{2}.
\]
\end{lemm}

\begin{proof}
Let $W = X \times X$ and $\alpha : W' \rightarrow W$ be the blow up of the diagonal. We denote its exceptional divisor by $E$. Let $C_P$ be a conic in the fiber at $P\in X_1$ passing through $P$. Then we have
\[
(H[2] - 2E).C_P = 0.
\]
As $P$ varies over $X_1$ $C_P$ forms a subvariety $D$ in $W'$ which is dominant to both $X_1$ and $X_2$. Thus our assertion follows from Lemma~\ref{lemm:coveringcurves}.

\end{proof}

Proposition~\ref{prop:cubic} shows that in general, $\delta(X, H)$ may not be $1/2$. 

\subsection{Local Tamagawa measures of conics in families}

Here we study the behavior of local Tamagawa measures of conics in a family.
Let $f: X \rightarrow S$ be a conic bundle defined over a number field $k$. Let $S^\circ$ be the complement of the discriminant locus $\Delta_f$ of $f$. 

Let $v$ be a place of $k$. We fix a $v$-adic metrization on $\mathcal O(K_X)$ and $\mathcal O(K_S)$. 
This induces a $v$-adic metrization on $\mathcal O(K_{X/S})$.
For each local $1$-form $ \mathrm dt \in \Omega^1_{X/S}$, one can define the local Tamagawa measure $\tau_{X_s, v}$ on a  conic $X_s$ for any $s \in S^\circ(k_v)$ by 
\[
\tau_{X_s, v}(U) = \int_{U} \frac{|\mathrm dt|}{\|\mathrm dt\|_v}
\]
which is independent of a choice of $\mathrm dt$.
\begin{lemm}
\label{lemm:smallballinfamily}
Suppose that $f : X \rightarrow S$ admits a rational section.
Then there exists $C>0$ such that for sufficiently small $T$, any $s \in S^\circ(k_v)$, $P\in X_s(k_v)$, we have
\[
C\mathrm{dist}_v(\Delta_f, f(P))T < \tau_{X_s, v}(\{Q \in X_s(k_v)\mid \mathrm{dist}_v(P, Q)<T\} ).
\]

\end{lemm}

\begin{proof}
We fix a rational section $S_0$ and an ample divisor $A$ on $S$.
Let $S_m = S_0 + mf^*A$. Now $f_*(\mathcal O(S_0)) \otimes \mathcal O(mA)$ is globally generated for $m \gg 0$. Using this for each point $p\in S$, one may find a rational section $S_p \sim S_m$ such that $S_p$ is a local section in a neighborhood of the point $p \in S$.
By the definition of conic bundles one can embed $f: X \rightarrow S$ into a projective bundle $\mathbb P(\mathcal E)$. Take a finite open affine covering $\{U_i\}$ of $S$ so that over $U_i$, $\mathbb P(\mathcal E)|_{U_i}$ is trivialized, i.e., isomorphic to $U_i \times \mathbb P^2$.
Taking a finer finite open covering, we may assume that $f$ admits a local section $S_i$ over $U_i$. By taking a finer finite open covering and applying a change of coordinates, one can assume that the local section $S_i$ corresponds to $(1:0:0)$ in $\mathbb P^2$. Moreover we may assume that the tangent line of $X_s$ at $(1:0:0)$ is given by $x_1 = 0$. Let $A_j \, (j = 0,1,2)$ be the standard affine charts of $\mathbb P^2$ and we define $V_{i, j} = f^{-1}(U_i) \cap (U_i \times A_j)$ which is affine. 

Now we take a finite $v$-adic open covering $B_l$ of $X(k_v)$ such that $\overline{B}_l$ is contained in some $V_{i, j}(k_v)$. Then on $B_l$ there exists a positive constant $C_1$ such that for any $(s, P), (s, Q)\in B_l \subset U_i(k_v) \times A_j(k_v)$ such that
\begin{align}
\label{bounded}
\mathrm{dist}_v(P, Q) \leq C_1\max\{|x_j(P) - x_j(Q)|_v, |y_j(P) - y_j(Q)|_v\},
\end{align}
where $x_j, y_j$ is the coordinates of $A_j$.

Now we are going to parametrize conics in the family.
By our construction, $f^{-1}(U_i) \subset U_i \times \mathbb P^2$ is defined by the following equation:
\[
d(s)y^2 + f(s)z^2 + 2xy + 2e(s)yz = 0,
\]
where $d, f, e$ are functions on $U_i$. Note that the discriminant locus $\Delta_f$ is defined by $f = 0$ and it is a smooth divisor by our assumption. After further simplifications, we may assume that the equation is given by
\[
f(s)z^2 + 2xy = 0.
\]

Lines $uy - vz = 0$ passing through $(1:0:0)$ are parametrized by $(u:v) \in \mathbb P^1$. Then the rational parameterization of conics is given by
\[
( fu^2 :-2v^2 : -2uv).
\]
In particular, any smooth conic $X_s$ over $s\in U_i(k_v)$ is covered by $V_{i, 0}$ and $V_{i, 1}$. Also note that while this rational parametrization is not valid along singular fibers, a rational map mapping $(s, P) \in f^{-1}(U_i)$ to $(u(P):v(P)) \in \mathbb P^1$ is a well-defined morphism.

Suppose that $B_l$ is contained in $V_{i, 0}$. The inequality~(\ref{bounded}) shows that there exists $C_2 > 0$ such that any $(s, P), (s, Q)\in B_l \subset U_i(k_v) \times A_j(k_v)$
\[
\mathrm{dist}_v(P, Q) < C_2 \mathrm{dist}_v(\Delta_f, f(P))^{-1}|t(P)-t(Q)|_v
\]
where $t = v/u$ and $ \mathrm{dist}_v(\Delta_f, f(P))$ is the distant function of $\Delta_f$.

Suppose that $B_l$ is contained in $V_{i, 1}$. The inequality~(\ref{bounded}) shows that there exists $C_3 > 0$ such that any $(s, P), (s, Q)\in B_l \subset U_i(k_v) \times A_j(k_v)$
\[
\mathrm{dist}_v(P, Q) < C_3 |t(P)-t(Q)|_v
\]
where $t = u/v$.

Suppose that $B_l$ is contained in $V_{i, 2}$. The inequality~(\ref{bounded}) shows that there exists $C_4 > 0$ such that any $(s, P), (s, Q)\in B_l \subset U_i(k_v) \times A_j(k_v)$
\[
\mathrm{dist}_v(P, Q) < C_4|t(P)-t(Q)|_v
\]
where $t = v/u$. Now by arguing as in Lemma~\ref{lemm: smallball} our assertion follows.

\end{proof}

\begin{lemm}
\label{lemm:tamagawainfamily}
Let $f: X \rightarrow S$ be a conic bundle defined over a number field $k$ with a rational section.
Let $v$ be an archimedean place of $k$. We fix a $v$-adic metrization on $\mathcal O(K_X)$ and $\mathcal O(K_S)$. 
Then for any sufficiently small $\epsilon > 0$ there exists a constant $C_\epsilon > 0$ such that for any $s \in S^\circ(k_v)$ we have
\[
\tau_v(X_s) <C_\epsilon\mathrm{dist}_v(\Delta_f, s)^{1-\epsilon}
\]

\end{lemm}

\begin{proof}
This follows from the descriptions in the proof of Lemma~\ref{lemm:smallballinfamily} and an explicit computations of local Tamagawa numbers using the naive metrization.
\end{proof}

\subsection{Fano conic bundles: the anticanonical height}

In this section, we discuss upper bounds of Manin type for the anticanonical height of Fano conic bundles. 
Here is a theorem:

\begin{theo}
\label{theo:main}

Let $f: X \rightarrow S$ be a conic bundle defined over a number field $k$ with a rational section.
We assume that $X$ and $S$ are Fano. 
Let $W = X \times X$ and $W'$ be the blow up of $W$ along the diagonal with the exceptional divisor $E$. We denote each projection $W' \rightarrow X_i$ by $\pi_i$.
Let $\alpha, \beta$ be positive real numbers such that $2\alpha- 2\beta =1 $. We further make the following assumptions:
\begin{enumerate}
\item Weak Manin's conjecture for $(S, -K_S)$ holds,
\item for any component $V$ of the stable locus of 
\[
|-\alpha K_{X/S}[2] - \beta f^*K_S[2] -E|
\]
such that $V$ is not contained in $E$, one of projections $\pi_i|_V$ is not dominant.
\end{enumerate}
Then there exists a non-empty Zariski open subset $U \subset X$ such that for any $\epsilon > 0$ there exists $C = C_\epsilon >0$ such that
\[
N(U, -K_X, T) < CT^{2\alpha +  \epsilon}.
\]
\end{theo}

\begin{proof}
First of all note that the assumption (2) implies that $-2\alpha K_X + f^*K_S$ is big.
Indeed, it implies that $-\alpha K_{X/S}- \beta f^*K_S$ is big; otherwise the linear system of this divisor defines a non-trivial fibration up to a birational modification and the assumption (2) cannot be true.
Then note that we have $-K_{X/S} = -K_X + f^*K_S$ and $2\alpha - 2\beta = 1$ so that $-\alpha K_{X/S}- \beta f^*K_S = -\alpha K_X + \frac{1}{2}f^*K_S$ so our claim follows.

Let $v$ be a place of $k$ and fix $v$-adic metrizations on $\mathcal O(K_X)$ and $\mathcal O(K_S)$. Fix $\epsilon > 0$. Arguing as Theorem~\ref{theo:repulsion}, the assumption (2) implies that there exists $U  \subset X$ and $C$ such that for any $P, Q \in U(k)$ with $P \neq Q$  and $f(P) = f(Q) = s$ we have
\begin{align}
\label{eqn}
\mathrm{dist}_v(P,Q)  > C (H_{-K_{S}}(s))^{-2\beta}(H_{-K_{X/S}}(P)H_{-K_{X/S}}(Q))^{-\alpha}.
\end{align}
We define
\[
A_T = \{P \in U(k) \mid H_L(P) \leq T\}.
\]
For $P \in A_T\cap X_s$, we define the $v$-adic ball by
\[
B_T(P) = \{R \in X_s\cap U(k_v) \mid \mathrm{dist}_v(P, R)< \frac{1}{2}CT^{-2\alpha}H_{-K_{S}}(s)\}
\]
Then $\cup B_T(P)$ is disjoint because of (\ref{eqn}) and the triangle inequality. Note that after shrinking $U$, $T^{-2\alpha}H_{-K_{S}}(s)$ uniformly goes to $0$ as $T \rightarrow \infty$ for any $s \in S$ with $A_T\cap X_s\neq \emptyset$ because of the Northcott property of the height function associated to $-2\alpha K_X + f^*K_S$ which is big. Thus we have
\begin{align*}
\mathrm{dist}_v(\Delta_f, s)^{1-\epsilon }&\gg_\epsilon \tau_v(X_s) > \sum_{P \in A_T\cap X_s(k)} \tau_{X_s, v}(B_T(P) )\\ &\gg_\epsilon N(U\cap X_s, L, T)\mathrm{dist}_v(\Delta_f, s)T^{-2\alpha}H_{-K_{S}}(s)
\end{align*}
by Lemma~\ref{lemm:smallballinfamily} and Lemma~\ref{lemm:tamagawainfamily}. Let $m$ be a positive integer such that $mL - f^*\Delta_f$ is ample. We conclude that 
\[
N(U\cap X_s, L, T) \ll_\epsilon  T^{2\alpha}H_{-K_{S}}(s)^{-1}H_{\Delta_f, v}(s)^\epsilon \ll  T^{2\alpha + m\epsilon} H_{-K_{S}}(s)^{-1}.
\]
Since $-kK_X \geq -f^*K_S$ for some $k$, our assertion follows from the fact that $S$ satisfies Weak Manin's conjecture and Tauberian theorem.
Indeed, Weak Manin's conjecture for $S$ and Tauberian argument implies that
\[
\sum_{s \in S^\circ(k)} H_{-K_{S}}(s)^{-(1 + \epsilon)} < +\infty
\]
where $S^\circ \subset S$ is a some open subset of $S$. Let $U^\circ  = U \cap f^{-1}(S^\circ)$. Then one can conclude that
\[
N(U^\circ, -K_X, T) = \sum_{s \in S^\circ(k)}N(U\cap X_s, L, T) \leq T^{2\alpha + m\epsilon + \epsilon k} \sum_{s \in S^\circ(k)}H_{-K_{S}}(s)^{-1-\epsilon} \ll T^{2\alpha + m\epsilon + \epsilon k}.
\]
Thus our assertion follows.


\end{proof}

\subsection{Fano conic bundles: the non-anticanonical heights}

In this section, we discuss weak Manin's conjecture for non-anticanonical height functions in some cases:

\begin{theo}
\label{theo:main2}

Let $f: X \rightarrow S$ be a conic bundle defined over a number field $k$ with a rational section.
We assume that $X$ and $S$ are Fano. Let $L = -K_X -tf^*K_S$. 
We make the following assumptions:
\begin{enumerate}
\item Weak Manin's conjecture for $(S, -K_S)$ holds,
\item $t \geq \frac{1}{2\delta(X, -K_X)}$.
\end{enumerate}
Then for any $\epsilon > 0$ there exists a non-empty Zariski open subset $U = U(\epsilon) \subset X$ and  $C = C_\epsilon >0$ such that
\[
N(U, L, T) < CT^{2\delta(X, -K_X) +  \epsilon}.
\]
In particular when $\delta(X, -K_X) = 1/2$, Conjecture~\ref{conj:weakManin} holds for $(X, L)$ except independence of $U$ on $\epsilon$. 
\end{theo}

For such a height function, \cite{FL17} establishes Manin's conjecture when the base is the projective space using conic bundle structures. Our theorem is flexible in the sense that $S$ can be other Fano manifold other than the projective space. For example one may find a smooth Fano threefold with a conic bundle structure over the Hirzebruch surface $\mathbb F_1$ in Section~\ref{sec:Fanoconic}.

\begin{proof}
Let $v$ be a place and fix $v$-adic metrizations on $\mathcal O(K_X)$ and $\mathcal O(K_S)$. Fix $\epsilon > 0$. Theorem~\ref{theo:repulsion} implies that there exists $U  \subset X$ and $C$ such that for any $P, Q \in U(k)$ with $P \neq Q$ we have
\begin{align}
\label{eqn2}
\mathrm{dist}_v(P,Q)  > C(H_{-K_{X}}(P)H_{-K_{X}}(Q))^{-(\delta(X, -K_X) + \epsilon)}.
\end{align}
We define
\[
A_T = \{P \in U(k) \mid H_L(P) \leq T\}.
\]
For $P \in A_T\cap X_s$, we define the $v$-adic ball by
\[
B_T(P) = \{R \in X_s\cap U(k_v) \mid  \mathrm{dist}_v(P, R)< \frac{1}{2}CT^{-2(\delta(X, -K_X) + \epsilon)}(H_{-K_{S}}(s))^{2t(\delta(X, -K_X) + \epsilon)}\}
\]
Then $\cup B_T(P)$ is disjoint because of (\ref{eqn2}) and the triangle inequality. Note that  
$$T^{-2(\delta(X, -K_X) + \epsilon)}(H_{-K_{S}}(s))^{2t(\delta(X, -K_X) + \epsilon)}$$ uniformly goes to $0$ as $T \rightarrow \infty$ for any $s \in S$ with $A_T\cap X_s\neq \emptyset$ after shrinking $U$. Thus we have
\begin{align*}
\mathrm{dist}_v(\Delta_f, s)^{1-\epsilon }&\gg_\epsilon \tau_v(X_s) > \sum_{P \in A_T\cap X_s(k)} \tau_{X_s, v}(B_T(P) )\\ &\gg_\epsilon N(U\cap X_s, L, T)\mathrm{dist}_v(\Delta_f, s)T^{-2(\delta(X, -K_X) + \epsilon)}(H_{-K_{S}}(s))^{2t(\delta(X, -K_X) + \epsilon)}
\end{align*}
by Lemma~\ref{lemm:smallballinfamily} and Lemma~\ref{lemm:tamagawainfamily}. Let $m$ be a positive integer such that $mL - f^*\Delta_f$ is ample.  We conclude that 
\[
N(U\cap X_s, L, T)  \ll_\epsilon  T^{2(\delta(X, -K_X) + \epsilon) + m\epsilon} H_{-K_{S}}(s)^{-2t(\delta(X, -K_X) + \epsilon)}.
\]
Since $L \geq -tf^*K_S$, our assertion follows by arguing as Theorem~\ref{theo:main}.


\end{proof}

\begin{rema}
If $\delta(X, -K_X)$ is the minimum, then one can take $U$ to be independent of $\epsilon$.
\end{rema}

\section{$3$-dimensional Fano conic bundles}
\label{sec:Fanoconic}

In this section we list smooth $3$-dimensional Fano conic bundles and compute $\delta(X, -K_X)$ and the smallest $2\alpha$ satisfying the conditions of Theorem~\ref{theo:main}. Fano $3$-folds with Picard rank $\geq 2$ are classified by Mori-Mukai in \cite{MM81}, \cite{MM83}, and \cite{MM03}. We follow their classification.
We assume that our ground field is an algebraically closed field of characteristic $0$.
In our computations of $\delta(X, -K_X)$ and the minimum $2\alpha$ satisfying the conditions of Theorem~\ref{theo:main}, it is important to know a description of the nef cone of divisors of $X$. Such a description has been obtained in \cite{Mat95}. We freely use the results in this lecture note.

\subsection{Fano threefolds with Picard rank $2$}

According to \cite{MM81}, there are $36$ deformation types of smooth Fano $3$-folds with Picard rank $2$. Among them there are $16$ deformation types of smooth Fano $3$-folds which come with conic bundle structures. Since Fano $3$-folds have Picard rank $2$, these conic bundle structures are extremal contractions. Thus in these cases, a conic bundle structure comes with a rational section if and only if there is no singular fiber. Thus there are $7$ deformation types of smooth Fano $3$-folds which come with a conic bundle structure with a rational section. Here is the list of these Fano $3$-folds from Table 2 of \cite{MM81}:

\begin{center}
  \begin{tabular}{ c | c | l | c | c}
    \hline
   no & $(-K_X)^3$ & $X$ & $6\delta(X, -K_X)$ & $2\alpha$ \\ \hline
    $24$ & $30$ &  a divisor on $\mathbb P^2 \times \mathbb P^2$ of bidegree $(1, 2)$ &$ \leq 6$ &$ \leq 5$ \\ \hline
     $27$ & $38$ & the blow-up of $\mathbb P^3$ with center a twisted cubic &  $3$ &  $2$ \\ \hline
      $31$ & $46$ & the blow-up of $Q \subset \mathbb P^4$ with center a line on it  &$3$ & $5/3$\\ \hline
       $32$ & $48$ & a divisor on $\mathbb P^2 \times \mathbb P^2$ of bidegree $(1, 1)$  & $3$ & $5/2$ \\ \hline
        $34$ & $54$ &  $\mathbb P^1 \times \mathbb P^2$ & $3$& $5/3$ \\ \hline
         $35$ & 56 & $V_7 = \mathbb P(\mathcal O \oplus \mathcal O(1))$ over $\mathbb P^2$ & $3$ & $5/4$ \\ \hline
    $36$ & 62 & $\mathbb P(\mathcal O \oplus \mathcal O(2))$ over $\mathbb P^2$ &$3$ & $5/3$ \\
   
    \hline
  \end{tabular}
\end{center}
Here $Q \subset \mathbb P^4$ is a smooth quadric $3$-fold.
Note that no $34$, $35$, and $36$ are toric thus Manin's conjecture is known for these cases by \cite{BT-general} and \cite{BT-0}. \cite{BBS18} proves Manin's conjecture for an example of Fano $3$-folds of no $24$.

Let us illustrate the computation of $\delta(X, -K_X)$ and $\alpha > 0$ in some cases:

\begin{exam}[no $32$]
Let $W$ be a smooth divisor of $\mathbb P^2 \times \mathbb P^2$ of bidegree $(1, 1)$. We denote each projection by $\pi_i : W \rightarrow \mathbb P^2$ and let $H_i$ be the pullback of the hyperplane class via $\pi_i$. Then we have
\[
-K_X = 2H_1 + 2H_2.
\]
Since $H_1 + H_2$ is very ample, it follows that $\delta(X, -K_X) = 1/2$ by Lemma~\ref{lemm:generalconicbundle}.

Next we consider
\[
\alpha(2H_1 + 2H_2 - 3H_1) + \frac{2\alpha -1}{2} \cdot 3H_1 = \frac{4\alpha -3}{2}H_1 + 2\alpha H_2.
\]
Then when $\frac{4\alpha -3}{2} \geq 1$, i.e., $\alpha \geq 5/4$, the above divisor satisfies the assumptions of Theorem~\ref{theo:main}.
On the other hand, for each $P \in X$, let $C_P$ be a fiber of $\pi_2$ meeting with $P$. Then we have $C_P. H_1 = 1$ and $C_P .H_2 = 0$. Thus we have
\[
\left(\frac{4\alpha -3}{2}H_1 + 2\alpha H_2 \right). C_P = \frac{4\alpha -3}{2}.
\]
Thus by Lemma~\ref{lemm:coveringcurves}, we conclude that $\alpha = 5/4$ is the minimum value satisfying the assumptions of Theorem~\ref{theo:main}.

\end{exam}

\begin{exam}[no $31$]
\label{exam:no31}
 Let $X$ be the blow-up of $Q$ along a line. Then $X$ has Picard rank $2$ so it comes with two extremal contractions, one is a $\mathbb P^1$-bundle $\pi_1 :  X \rightarrow \mathbb P^2$, and another is a divisorial contraction $\pi_2 : X \rightarrow Q$. Let $H_i$ be the pullback of the hyperplane class via $\pi_i$. Then it follows from \cite[Theorem 5.1]{MM83} that
\[
-K_X = H_1 + 2H_2.
\]
Since $\delta(X, H_2) = \delta(Q, H) = 1$ and $\pi_2$ is birational, it follows that $\delta(X, -K_X) \leq 1/2$. Thus by Lemma~\ref{lemm:generalconicbundle}, $\delta(X, -K_X) = 1/2$ is proved.

Next we have
\[
\alpha(H_1 + 2H_2 - 3H_1) + \frac{2\alpha -1}{2} \cdot 3H_1 = 2\alpha H_2 + \frac{2\alpha -3}{2} H_1.
\]
Let $D$ be the exceptional divisor of $\pi_2$. Then we have $H_1 = H_2 -D$. Thus the above divisor becomes
\[
\frac{6\alpha -3}{2}H_2 - \frac{2 \alpha - 3}{2}D.
\]
Thus when $\frac{6\alpha -3}{2} \geq 1$ and $2 \alpha - 3 \leq 0$, i.e, $5/6 \leq \alpha \leq 3/2$ the assumption of Theorem~\ref{theo:main} holds. On the other hand let $\ell \subset X$ be the strict transform of a line on $Q$ not meeting with center of $\pi_2$. Then we have 
\[
\left(\frac{6\alpha -3}{2}H_2 - \frac{2 \alpha - 3}{2}D \right).\ell = \frac{6\alpha -3}{2}.
\]
Thus since such $\ell$ deforms to cover $X$, by arguing as Lemma~\ref{lemm:coveringcurves}, we conclude that $\alpha = 5/6$ is the minimum value satisfying the assumptions of Theorem~\ref{theo:main}.

\end{exam}

\subsection{Fano threefolds with Picard rank $3$}
According to \cite{MM81}, there are $31$ deformation types of smooth Fano $3$-folds with Picard rank $3$. It follows from \cite[p. 125, (9.1)]{MM83} that all such Fano $3$-folds come with a conic bundle structure except the blow-up of $\mathbb P^3$ along a disjoint union of a line and a conic. Again a conic bundle structure with singular fibers which is extremal never comes with a rational section. 
Note that if $X$ is a Fano conic bundle which does not admit a divisorial contraction to a Fano conic bundle of Picard rank $2$ with a rational section, then its extremal conic bundle structure admits singular fibers. For such a $3$-fold, one can conclude that it does not admit a rational section. This implies that there are $25$ deformation types of $3$ dimensional Fano conic bundles with a rational section. Here is the list of these Fano $3$-folds from Table 3 of \cite{MM81}:

\begin{center}
  \begin{tabular}{ c | l | c | c }
    \hline
   no &  $X$&$6\delta$ & $2\alpha$ \\ \hline
    3 &   a divisor on  $\mathbb P^1\times \mathbb P^1 \times \mathbb P^2$ of tridegree $(1, 1, 2)$ & $\leq 6$ &$\leq 5$  \\ \hline
     $5$ & the blow-up of $\mathbb P^1 \times \mathbb P^2$ with center a curve $C$ of bidegree $(5,2)$ &$\leq 6$ & $\leq 5$ \\ 
     &  such that the projection $C \rightarrow \mathbb P^2$ is an embedding & &\\ \hline
      $7$ &  the blow up of $W$ (no $32$)&  $\leq 4$ & $\leq 3$  \\  & with center an intersection of two members of $|-\frac{1}{2}K_W|$& &  \\ \hline
       $8$ &  a member of the linear system $|p_1^*g^*\mathcal O(1)\otimes p_2 \mathcal O(2)|$ on $\mathbb F_1\times \mathbb P^2$  where & $\leq 6$& $\leq 5$
       \\ & $p_i$ is the projection to each factor and $g:\mathbb F_1 \rightarrow \mathbb P^2$ is the blowing up & & \\ \hline
        $9$ &   the blowing up of the cone $W_4 \subset \mathbb P^6$ over the veronese surface $R_4 \subset \mathbb P^5$  & $3$ &$\leq 7/5$  \\  & with center a disjoint union of the vertex and quartic in $R_4 = \mathbb P^2$ & & \\ \hline
         $11$ &   the blow up of $V_7$ (no $35$)& $3$ & $5/2$  \\ &  with center an intersection of two members of $|-\frac{1}{2}K_{V_7}|$ & & \\ \hline
         12 &  the blow up of $\mathbb P^3$ with center a disjoint union of a line and a twisted cubic & $3$ & $8/3$  \\ \hline
          13 & the blow up of $W \subset \mathbb P^2 \times \mathbb P^2$ with center a curve $C$ of bidegree $(2, 2)$ on it & $3$ & $\leq 5/2$ \\  & such that each projection from $C$ to $\mathbb P^2$ is an embedding & & \\ \hline
          14 &  the blow up of $\mathbb P^3$ with center a disjoint union of a point and a plane cubic & $\leq 6$& $\leq 9/5$ \\ \hline
          15 &  the blow up of $Q \subset \mathbb P^4$ with center a disjoint union of a line and a conic &  $3$&  $5/2$ \\ \hline
          16 &   the blow up of $V_7$ with center the strict transform of a twisted cubic & $3$ & $5/2$
          \\ & passing through the center of the blow up $V_7 \rightarrow \mathbb P^3$ & & \\ \hline
          17 &  a smooth divisor on $\mathbb P^1 \times \mathbb P^1 \times \mathbb P^2$ of tridegree $(1, 1, 1)$ &$3$ &$5/2$ \\ \hline
           19 &   the  blow up of $Q \subset\mathbb P^4$ with center two points which are not colinear & $3$ &$5/3$ \\ \hline
            20 &  the  blow up of $Q \subset\mathbb P^4$ with center a disjoint union of two lines &$3$ & $5/2$ \\ \hline
          21 &  the  blow up of $\mathbb P^1 \times \mathbb P^2$ with center a curve of bidegree $(2, 1)$ & $3$ &$5/2$ \\ \hline
           22 &  the  blow up of $\mathbb P^1 \times \mathbb P^2$ with center a conic in $\{t\} \times \mathbb P^2$ & $3$ &$5/3$  \\ \hline
           23 &   the blow up of $V_7$ with center the strict transform of a conic & $3$ & $5/3$
          \\  & passing through the center of the blow up $V_7 \rightarrow \mathbb P^3$& & \\ \hline
             24 &  the fiber product $W\times_{\mathbb P^2}\mathbb F_1$, where  $W$ is no $32$ & $3$ & $5/2$
             \\ &  $W \rightarrow \mathbb P^2$ is the $\mathbb P^1$-bundle and $\mathbb F_1 \rightarrow \mathbb P^2$ is the blowing up & & \\ \hline
              25 & $\mathbb P(\mathcal O(1,0) \oplus \mathcal O(0,1))$ over $\mathbb P^1 \times \mathbb P^1$ &$3$ & $3/2$\\ \hline
               26 &  the blow up of $\mathbb P^3$ with center a disjoint union of a point and a line &  $3$& $5/3$ \\ \hline
               27 &  $\mathbb P^1 \times \mathbb P^1 \times \mathbb P^1$  & $3$ & $2$ \\ \hline
                   28 &  $\mathbb P^1 \times \mathbb F_1$ & $3$ & $2$ \\ \hline
                   29 & the blow up of $V_7$ with center a line  & $3$ &$\leq 7/5$ \\  & on the exceptional set $D = \mathbb P^2$ of the blow up $V_7 \rightarrow \mathbb P^3$ & & \\ \hline
                     30 &   the blow up of $V_7$ with center the strict transform of a line &$3$ & $4/3$
          \\  & passing through the center of the blow up $V_7 \rightarrow \mathbb P^3$ & & \\ \hline
    31 & $\mathbb P(\mathcal O \oplus \mathcal O(1,1))$ over $\mathbb P^1 \times \mathbb P^1$ &$3$ & $\leq 4/3$ \\ 
   
    \hline
  \end{tabular}
\end{center}

Note that no 24-31 are toric, thus Manin's conjecture is known for these cases by \cite{BT-general} and \cite{BT-0}. Let us demonstrate the computation of $\delta(X, -K_X)$ and $\alpha$ in some cases:

\

\begin{exam}[no $23$]
\label{exam:no23}
Let $X$ be a Fano $3$-fold of no $23$. Then $X$ admits a divisorial contraction $\beta: X \rightarrow V_7$ with the exceptional divisor $D_1$. The Fano $3$-fold $V_7$ admits two extremal contractions: one is a $\mathbb P^1$-bundle $\pi_1 : V_7 \rightarrow \mathbb P^2$ and another is the blow down $V_7 \rightarrow \mathbb P^3$. We denote the pullback of the hyperplane class via $\pi_i$ by $H_i$. 
One can conclude that the only conic bundle structure on $X$ is $\pi_1 \circ \beta$.
It follows from \cite[Theorem 5.1]{MM83} that
\[
-K_{V_7} = 2H_1 + 2H_2.
\]
Thus we have 
\[
-K_X = 2\beta^*H_1 + 2\beta^*H_2 -D_1.
\]
Since $\beta^*H_1 - D_1$ is effective and the morphism associated to $|H_2|$ is birational, one can conclude that $\delta(X, -K_X) = 1/2$ by Lemma~\ref{lemm:generalconicbundle}.

Let $D_2$ be the strict transform of the exceptional divisor of $\pi_2$. Then we have
\[
\alpha(2\beta^*H_1 + 2\beta^*H_2 -D_1 - 3\beta^*H_1) + \frac{2\alpha-1}{2}\cdot 3\beta^*H_1 = 2\alpha \beta^*H_2 - \alpha D_1  + \frac{4\alpha -3}{2}\beta^*H_1
\]
Since we have $\beta^*H_2 = \beta^*H_1 + D_2$, the above divisor becomes
\[
\beta^*H_2 + (2\alpha-1)D_2+ \frac{8\alpha -5}{2}\beta^*H_1- \alpha D_1 
\]
Since $\beta^*H_1-D_1 \geq 0$, in the case of $ \frac{8\alpha -5}{2}\geq\alpha$, i.e., $\alpha \geq 5/6$, the above divisor satisfies the assumption of Theorem~\ref{theo:main}. On the other hand let $\ell$ be the strict transform of a line meeting with the center of $D_1$. When $\alpha = 5/6$, we have
\[
\left(\beta^*H_2 + (2\alpha-1)D_2+ \frac{8\alpha -5}{2}\beta^*H_1- \alpha D_1 \right).\ell = 1
\]
Thus since such $\ell$ deforms to cover $X$, by arguing as Lemma~\ref{lemm:coveringcurves}, we conclude that $\alpha = 5/6$ is the minimum value satisfying the assumption of Theorem~\ref{theo:main}.

\end{exam}

\begin{exam}[no $12$]
\label{exam:no12}
Let $X$ be a Fano $3$-fold of no $12$. Then it admits a conic bundle structure $\pi_1 :  X \rightarrow \mathbb P^2$, a birational morphism $\pi_2 : X \rightarrow \mathbb P^3$, and a del Pezzo fibration $\pi_3 : X \rightarrow \mathbb P^1$. Let $H_i$ be the pullback of the hyperplane class via $\pi_i$. Then we have
\[
-K_X = H_1 + H_2 + H_3
\]
Since $|H_2|$ defines a birational morphism to $\mathbb P^3$ and $|H_1 + H_3|$ defines a birational morphism to $\mathbb P^2 \times \mathbb P^1$ we can conclude that $\delta(X, -K_X) \leq 1/2$. Thus by Lemma~\ref{lemm:generalconicbundle}, we have $\delta(X, -K_X) = 1/2$.
Next we consider the following divisor
\[
\alpha(H_1 + H_2 + H_3 - 3H_1) + \frac{2\alpha - 1}{2} \cdot 3H_1 = (3 \alpha - 3)H_2 + \alpha H_3 - \frac{2\alpha -3}{2}D_1 
\]
where $D_1$ is the exceptional divisor of $\pi_2$ whose center is a twisted cubic. When $3\alpha - 3 \geq 1$ and $2\alpha - 3 \leq0$, i.e., $4/3 \leq \alpha \leq 3/2$ the assumptions of Theorem~\ref{theo:main} holds. On the other hand let $\ell$ be the strict transform of a general line meeting with the line which is the center of $\pi_2$. Then we have
\[
\left((3 \alpha - 3)H_2 + \alpha H_3 - \frac{2\alpha -3}{2}D_1 \right).\ell = 3 \alpha -3.
\]
Thus since such $\ell$ deforms to cover $X$, by arguing as Lemma~\ref{lemm:coveringcurves}, we conclude that $\alpha = 4/3$ is the minimum value satisfying the assumption of Theorem~\ref{theo:main}.

\end{exam}

\subsection{Fano threefolds with Picard rank $4$ or $5$}

For Fano $3$-folds in this range, all of them admit conic bundle structures with a rational section. Here is the list of Fano $3$-folds of Picard rank $4$ from \cite[Table 4]{MM81} as well as \cite{MM03}:

\begin{center}
{\bf Fano $3$-folds with Picard rank $4$}
  \begin{tabular}{ c | c | l |c|c}
    \hline
   no &  $X$ &$6\delta(X, -K_X)$&$2\alpha$ \\ \hline
    $1$ &  a smooth divisor on $(\mathbb P^1)^4$ of multidegree $(1, 1, 1, 1)$ & $3$&$3$\\ \hline
     $2$ & the blow up of the cone over a quadric surface $S \subset \mathbb P^3$ &$\leq 6$& $\leq 2$ \\
       & with center a disjoint union of the vertex and an elliptic curve on $S$ && \\ \hline
      $3$ &  the blow-up of $\mathbb P^1 \times\mathbb P^1 \times\mathbb P^1$ with center a curve of tridegree $(1, 1, 2)$ & $3$ &$\leq 2$ \\ \hline
       $4$ &  the blow up of $Y$ (no $19$ Table $3$) & $\leq 6$&$\leq 2$ \\ &  with center the strict transform of a conic passing through $p$ and $q$ &&\\ \hline
        $5$ &   the blow up of $\mathbb P^1 \times \mathbb P^2$& $\leq 6$& $2$ \\ &  with center two disjoint curves of bideree $(2, 1)$ and $(1, 0)$ &&\\ \hline
         $6$ &  the blow up of $\mathbb P^3$ with center three disjoint lines & $3$ & $2$\\ \hline
         $7$ &  the blow up of $W \subset \mathbb P^2 \times \mathbb P^2$ &$3$& $5/2$ \\ &  with center two disjoint curves of bidegree $(0,1)$ and $(1, 0)$ && \\ \hline
    $8$ &  the blow-up of $\mathbb P^1 \times\mathbb P^1 \times\mathbb P^1$ with center a curve of tridegree $(0, 1, 1)$ & $3$& $2$ \\ \hline
     $9$ &  the blow up of $Y$ (no $25$ Table 3)& $3$& $2$ \\ &  with center an exceptional line of the blowing up $Y \rightarrow \mathbb P^3$&& \\ \hline
      $10$ & $\mathbb P^1\times S_7$ & $3$& $2$\\ \hline
       $11$ &  the blow up of $\mathbb P^1 \times \mathbb F_1$ with center $t\times e$ &$3$  & $2$
       \\ &  where $t \in \mathbb P^1$ and $e$ is an exceptional curve on $\mathbb F_1$&& \\ \hline
        $12$ &  the blow up of $Y$ (no $33$ Table 2) &$3$& $2$
        \\ &  with center two exceptional lines of the blowing up $Y \rightarrow \mathbb P^3$& & \\ \hline
        $13$ &  the blow-up of $\mathbb P^1 \times\mathbb P^1 \times\mathbb P^1$ with center a curve of tridegree $(1, 1, 3)$ & $\leq 6$ & $\leq 3$\\ \hline
  \end{tabular}
\end{center}
Here $S_7$ is a smooth del Pezzo surface of degree $7$. Note that no 9-12 are toric, so Manin's conjecture is known for these cases by \cite{BT-general} and \cite{BT-0}. Again let us illustrate the computation of $\delta(X, -K_X)$ and $\alpha$ in some cases:
\begin{exam}[no $6$]
\label{exam:no6}
Let $X$ be a Fano $3$-fold of no $6$. Then $X$ admits three del Pezzo fibrations $\pi_i : X \rightarrow \mathbb P^1$. It also admits a birational morphism $\pi : X \rightarrow \mathbb P^3$. Let $H_i$ be the pullback of the hyperplane class via $\pi_i$. Let $H$ be the pullback of the hyperplane class via $\pi$. Then we have
\[
-K_X = H + H_1 + H_2 + H_3.
\]
Since both $|H|$, $|H_1 + H_2 + H_3|$ are birational, we conclude that $\delta(X, -K_X) = 1/2$.

Next any conic bundle structure on $X$ is given by $|H_i + H_j|$ where $i \neq j$. So we look at the conic bundle structure defined by $|H_1 + H_2|$. We consider
\[
\alpha(H + H_3 - H_1 - H_2) + (2\alpha -1)(H_1 + H_2) = \alpha H + \alpha H_3 + (\alpha - 1)H_1 + (\alpha -1)H_2.
\]
From this description we may conclude that $\alpha = 1$ satisfies the assumptions of Theorem~\ref{theo:main}. On the other hand, by looking at the strict transform $C$ of a line such that $H_3.C = 0$, we may conclude that $\alpha = 1$ is the minimum value satisfying the assumptions of Theorem~\ref{theo:main}.
\end{exam}

\begin{exam}[no $4$]
Let $V_7$ be the blow-up of $\mathbb P^3$ at a point $p$. It admits two extremal rays and we denote each contraction morphism by $\pi_1 : V_7 \rightarrow \mathbb P^2$ and $\pi_2 : X \rightarrow \mathbb P^3$. Let $H_i$ be the pullback of the hyperplane via $\pi_i$. Let $X'$ be the blow-up of a fiber of $\pi_1$ which is the strict transform of a line $\ell$ passing through $p$. It admits a conic bundle structure over $\mathbb F_1$. Let $X$ be the the blow up of $X'$ along the strict transform of a conic $C_0$ not meeting with $\ell$. Then $X$ is a smooth Fano $3$-fold of no $4$.
We denote the strict transform of the exceptional divisor of $X' \rightarrow V_7$ by $D_1$ and the exceptional divisor of $X \rightarrow X'$ by $D_2$. Then we have
\[
-K_X = 2H_1 + 2H_2 - D_1 - D_2.
\]
Since $2H_1 + H_2- D_1 -D_2$ is linearly equivalent to an effective divisor, it follows that $\delta(X, -K_X) \leq 1$.

Next we consider
\[
\alpha(-K_X - (3H_1 - D_1)) + \frac{2\alpha -1}{2}(3H_1-D_1) = 2\alpha H_2 - \alpha D_2 + \frac{4\alpha -3}{2} H_1 - \frac{2\alpha -1}{2}D_1.
\]
Since $H_2-D_2$ and $H_1 - D_1$ are effective, it follows that $\alpha = 1$ satisfies the assumptions of Theorem~\ref{theo:main}.
\end{exam}

Finally we discuss the case of Picard rank $5$. Here is the list of Fano $3$-folds of Picard rank $5$ from \cite[Table 5]{MM81}

\begin{center}
{\bf Fano $3$-folds with Picard rank $5$}
  \begin{tabular}{ c | c | l |c | c }
    \hline
   no &  $X$  &$6\delta(X, -K_X)$ &$2\alpha$ \\ \hline
 $1$ &  the blow up of $Y$ (no $29$ Table 2) & $3$& $\leq 2$
        \\ &  with center three exceptional lines of the blowing up $Y \rightarrow  Q$&  &  \\ \hline
      $2$ &  the blow up of $Y$ (no $25$ Table 3)& $3$ & $\leq 2$
        \\ &  with center two exceptional lines $\ell$ and $\ell'$ of the blowing up $\phi: Y \rightarrow \mathbb P^3$ && \\ &  such that $\ell$ and $\ell'$ lie on &&\\ &  the same irreducible component of the exceptional set for $\phi$&& \\ \hline
      $3$ & $\mathbb P^1\times S_6$ &$3$&$2$ \\ \hline
  \end{tabular}
\end{center}
Here $S_6$ is a smooth del Pezzo surface of degree $6$. Note that no 2, 3 are toric, so Manin's conjecture is known by \cite{BT-general} and \cite{BT-0}.

\subsection{Fano threefolds with Picard rank $\geq 6$}

We start this section by the following theorem of Mori-Mukai:
\begin{theo}{\cite[Theorem 1.2]{MM83}}
Let $X$ be a smooth Fano $3$-fold and we denote its Picard rank by $\rho(X)$.
Suppose that $\rho(X) \geq 6$. Then $X$ is isomorphic to $\mathbb P^1\times S_{11-\rho(X)}$
where $S_d$ is a smooth del Pezzo surface of degree $d$. 
\end{theo}

Thus in this section, we study the product of a smooth del Pezzo surface $S_d$ with $\mathbb P^1$. Note that weak Manin's conjecture for the product follows as soon as weak Manin's conjecture is known for $S_d$ by \cite{FMT89}, so we omit the discussion of $\alpha$ in this section. 


\begin{prop}
\label{rema:product}
Let $X = \mathbb P^1 \times S$ where $S$ is a smooth del Pezzo surface of degree $d$ with $1 \leq d \leq 8$. Then we have $\delta(X, -K_X) = \delta(S, -K_S)$.
\end{prop}
\begin{proof}
Let $W_X = X \times X$ and $\alpha : W_X' \rightarrow W_X$ be the blow up of the diagonal. We use the same notation for $S$ as well. Let $H_1$ be the pullback of the ample generator via $p_1 : X \rightarrow \mathbb P^1$ and let $H_2$ be the pullback of the anticanonical divisor via $p_2 : X \rightarrow S$. Then the anticanonoical divisor of $X$ is
\[
-K_X = 2H_1 + H_2.
\]
Fix $\epsilon > 0$ and we consider
\[
-(\delta(S, -K_S) + \epsilon)K_X[2] -E = 2( \delta(S, -K_S)+\epsilon) H_1[2] + \epsilon H_2[2] + \delta(S, -K_S)H_2[2]-E.
\]
Since $2( \delta(S, -K_S)+\epsilon) H_1[2] + \epsilon H_2[2]$ is semi-ample, the stable locus of $-(\delta(S, -K_S) + \epsilon)K_X[2] -E$ is contained in the stable locus of $\delta(S, -K_S)H_2[2]-E$. The possible dominant components of the stable locus of $\delta(S, -K_S)H_2[2]-E$ are $E$ and the strict transform of 
\[
\mathbb P^1\times \mathbb P^1 \times \Delta_S
\]
where $\Delta_S$ is the diagonal of $W_S$.
Next we consider
\[
-(\delta(S, -K_S) + \epsilon)K_X[2] -E = 2\epsilon H_1[2] + (\delta(S, -K_S) + \epsilon )H_2[2] + 2\delta(S, -K_S)H_1[2]-E
\]
Since $2\epsilon H_1[2] + (\delta(S, -K_S) + \epsilon )H_2[2] $ is again semi-ample, the stable locus of $-(\delta(S, -K_S)  + \epsilon)K_X[2] -E$ is contained in the stable locus of $2\delta(S, -K_S) H_1[2]-E$. Since $\delta(S, -K_S) \geq 1/2$, it follows that the stable locus of $2\delta(S, -K_S)H_1[2]-E$ is contained in the strict transform $Z$ of
\[
 \Delta_{\mathbb P^1} \times  W_S \subset W_X
\]
where $\Delta_{\mathbb P^1} $ is the diagonal of $\mathbb P^1 \times \mathbb P^1$.
The variety $Z$ is isomorphic to $\mathbb P^1 \times W_S'$. 
From this one may conclude that $\delta(X, -K_X) \leq \delta(S, -K_S)$ by taking the intersection of two loci. 

On the other hand the discussion in Section~\ref{subsec:delpezzo} shows that in each case there are curves $C$ on $W_S'$ such that (i) $C$ deforms to dominate both $S_i$ and (ii) $(-\delta(S, -K_S)K_S[2]-E).C = 0$. Thus we conclude that $\delta(X, -K_X) \geq \delta(S, -K_S)$. Thus our assertion follows.

\end{proof}

\bibliographystyle{alpha}
\bibliography{Vojta_BM}

\end{document}